\documentclass[a4paper,reqno]{amsart}
\usepackage{ 
	amsmath, 
	amssymb, 
	amsthm, 
	amscd, 
	amstext, 
	verbatim, 
	enumerate, 
	mathtools,
	color}

\usepackage{hyperref}
\usepackage{comment}
\usepackage[all]{xy}
\usepackage[utf8]{inputenc} %
\usepackage{lmodern}        %
\usepackage[T1]{fontenc}    %
\usepackage{hyperref}       %
\usepackage{url}            %
\usepackage{booktabs}       %
\usepackage{amsfonts}       %
\usepackage{nicefrac}       %
\usepackage{microtype}      %
\usepackage{xcolor}         %
\usepackage{mathtools}		%
\usepackage[disable]{todonotes}	%
\usepackage{aliascnt} 		%
\usepackage{enumitem}		%
\usepackage{graphicx}       %
\usepackage{caption}        %
\usepackage{subcaption}     %
\usepackage{float} 			%
\usepackage{forest}			%
\usepackage{mathrsfs} 		%
\setlist[enumerate,1]{label={(\roman*)}} %
\makeatletter
\newcommand{\myitem}[1]{%
	\item[#1]\protected@edef\@currentlabel{#1}%
}
\makeatother

\usepackage{tikz}
\usetikzlibrary{arrows.meta}
\usepackage{tikz-cd}

\hypersetup{
	colorlinks,
	linkcolor={blue!80!black},
	citecolor={green},
	urlcolor={blue!80!black}
}

\theoremstyle{plain}
\newtheorem{theorem}{Theorem}[section]

\newaliascnt{propCt}{theorem}
\newtheorem{prop}[propCt]{Proposition}
\aliascntresetthe{propCt}

\newaliascnt{lemmaCt}{theorem}
\newtheorem{lemma}[lemmaCt]{Lemma}
\aliascntresetthe{lemmaCt}

\newaliascnt{corCt}{theorem}
\newtheorem{cor}[corCt]{Corollary}
\aliascntresetthe{corCt}

\theoremstyle{definition}
\newaliascnt{defiCt}{theorem}
\newtheorem{defi}[defiCt]{Definition}
\aliascntresetthe{defiCt}

\newaliascnt{remCt}{theorem}
\newtheorem{rem}[remCt]{Remark}
\aliascntresetthe{remCt}

\newaliascnt{qstCt}{theorem}

\aliascntresetthe{qstCt}

\newaliascnt{pbmCt}{theorem}

\aliascntresetthe{pbmCt}

\newaliascnt{exaCt}{theorem}
\newtheorem{exa}[exaCt]{Example}
\aliascntresetthe{exaCt}

\theoremstyle{plain}
\newtheorem{alphthm}{Theorem}			%

\newaliascnt{alphqstCt}{alphthm}

\aliascntresetthe{alphqstCt}

\newaliascnt{alphcorCt}{alphthm}

\aliascntresetthe{alphcorCt}

\newaliascnt{alphpropCt}{alphthm}

\aliascntresetthe{alphpropCt}

\newaliascnt{alphproblemCt}{alphthm}
\newtheorem{alphproblem}[alphproblemCt]{Problem}
\aliascntresetthe{alphproblemCt}

\newtheorem*{claim}{Claim}

\numberwithin{equation}{section}

\newcommand{\N}{\mathbb N}
\newcommand{\Z}{\mathbb Z}
\newcommand{\C}{\mathbb C}

\newcommand{\pr}{\mathrm{pr}}

\DeclareMathOperator{\supp}{supp}
\DeclareMathOperator{\id}{id}

\begin{document}
\title[Groupoids, partial actions, inner amenability, Kirchberg algebras]{On groupoids beyond partial actions, inner amenability, and models for Kirchberg algebras}
\author[Alcides Buss]{Alcides Buss}
\address[Alcides Buss]{Departamento de Matem\'atica, Universidade Federal de Santa Catarina, 88.040-900
Florian\'opolis-SC, Brazil}
	\email{alcides.buss@ufsc.br}
	\urladdr{http://mtm.ufsc.br/~alcides/}

\author[Julian Kranz]{Julian Kranz}
\address[Julian Kranz]{Universit\"at M\"unster, Mathematisches Institut, Einsteinstr. 62, 48149 M\"unster, Germany}
\email{julian.kranz@uni-muenster.de}
\urladdr{https://sites.google.com/view/juliankranz/}


\begin{abstract}

We construct the first explicit examples of locally compact Hausdorff
étale groupoids that are not inner amenable and that do not arise as
transformation groupoids associated to partial actions of discrete groups.
This answers questions of Anantharaman--Delaroche and Exel.
Our examples include all Higson--Lafforgue--Skandalis groupoids associated
to non-amenable residually finite groups, as well as their
principal variants constructed by Alekseev--Finn--Sell. These can be chosen
to be second countable, ample, and in the latter case even
principal.

We also show that large classes of Deaconu--Renault groupoids with connected unit space do not arise from partial actions of discrete groups, including cases whose
$C^*$-algebras are Kirchberg algebras in the UCT class.
We contrast this with the totally disconnected case by giving ample transformation groupoid models for all unital Kirchberg algebras in the UCT class as well as many higher rank graph algebras.
Finally, we characterize precisely when coarse groupoids arise from partial
actions of discrete groups in terms of coarse embeddings into groups.
\end{abstract}

\maketitle
\setcounter{tocdepth}{1}
\tableofcontents

\section{Introduction}

\'Etale groupoids have become one of the central organizing structures in
the modern theory of $C^*$-algebras. Since Renault's seminal work
\cite{Renault1980}, it has been understood that a large and flexible class of
$C^*$-algebras can be realized as $C^*$-algebras of locally compact Hausdorff
\'etale groupoids, and that the presence of a Cartan subalgebra corresponds
precisely to such a realization.
 
More precisely, Renault showed in \cite{Renault:Cartan} that every (separable) Cartan
pair $(A,B)$ arises from a twisted étale groupoid, in the sense that
\[
(A,B)\cong (C_r^*(G,\Sigma),C_0(G^{(0)}))
\]
for a (second countable) topologically principal twisted étale groupoid $(G,\Sigma)$.
Thus the problem of understanding $C^*$-algebras with Cartan subalgebras
can be approached through the study of étale groupoids and vice versa.

Recent advances in the Elliott classification program\footnote{We refer to \cite{Winter2018,White2023} for an overview of the Elliott classification program.} have further reinforced
the importance of this viewpoint. In particular, Barlak--Li \cite{Barlak-Li-I} (building on \cite{Tu1999}) proved that
every nuclear $C^*$-algebra with a Cartan subalgebra satisfies the Universal Coefficient Theorem (UCT) of \cite{Rosenberg1987}, while conversely Li \cite{Li:Classifiable-Cartan} (building on \cite{Spielberg2007,CFH-Kirchberg}) established that every `Elliott--classifiable' $C^*$-algebra has a Cartan subalgebra.  In particular, one of the main open problems in $C^*$-algebra theory -- whether every nuclear $C^*$-algebra satisfies the UCT -- can, in principle, be addressed by showing that every nuclear $C^*$-algebra admits a twisted étale groupoid model.

Presenting abstract $C^*$-algebras via \'etale groupoids has led to tremendous progress in understanding analytic approximation properties of $C^*$-algebras, most notably, nuclearity and exactness. Nuclearity of \'etale groupoid $C^*$-algebras can be described dynamically in terms of \emph{amenability} of the underlying groupoid \cite{AnantharamanDelaroche2000,Buss2026} -- an approximation property whose consequences in operator algebras and noncommutative geometry are hard to overestimate \cite{Connes1981,Higson2001,Tu1999,Gabe2024}. 
Similarly, exactness of group $C^*$-algebras was characterized by Ozawa \cite{Ozawa2000} in terms of \emph{amenability at infinity} or Yu's \emph{property A} \cite{Yu2000} of the underlying groupoid. Ozawa's seminal result has similarly striking consequences including the resolution of the Novikov conjecture for such groups \cite{Guentner2002}.
In a recent preprint, Anantharaman--Delaroche \cite{AnantharamanDelaroche} (see also \cite{Anantharaman2023} for an accessible overview) generalized Ozawa's theorem to second countable Hausdorff \'etale groupoids under the additional assumption of a technical condition called \emph{inner amenability}. 
Inner amenability has been established for many examples including all discrete groups, transformation groupoids associated to group actions, and more generally transformation groupoids associated to \emph{partial group actions} with clopen domains (see \cite[Corollary~5.18]{AnantharamanDelaroche}). 
Its validity in full generality was left as the main open problem in \cite{AnantharamanDelaroche}:
\begin{alphproblem}[Anantharaman--Delaroche]\label{prob-inneramenable}
	Is every locally compact Hausdorff \'etale groupoid inner amenable?
\end{alphproblem}

Transformation groupoids arising from partial actions of discrete groups on locally compact Hausdorff spaces form a particularly rigid and well-understood subclass of étale groupoids.
Given a partial action of a discrete group $\Gamma$
on a locally compact Hausdorff space $X$ by homeomorphisms between open subsets of $X$, the associated
transformation groupoid $\Gamma \ltimes X$
is Hausdorff étale and provides a crossed product description
\[
C^*_r(\Gamma\ltimes X) \cong C_0(X)\rtimes_r \Gamma,
\]
see \cite{Renault1980,Exel1994,McClanahan1995,Abadie2004}.
Large classes of $C^*$-algebras -- including all graph $C^*$-algebras --
admit such realizations via partial crossed products, see \cite{Li2017,Exel-Starling}.
This naturally leads to the following fundamental structural problem communicated to us by Exel \cite{Exelprivate}:

\begin{alphproblem}[Exel]\label{prob-transformation}
Are all locally compact Hausdorff \'etale groupoids isomorphic to transformation groupoids associated to partial actions of discrete groups?
\end{alphproblem}

While structural characterizations of transformation groupoids are known \cite{deCastroKang2024},
the existence of explicit examples of Hausdorff étale groupoids that fail
to admit such a realization appears not to have been documented in the
literature.

The purpose of this paper is to solve both problems by answering them negatively. 
The counterexamples to the first problem are groupoids introduced by Higson--Lafforgue--Skandalis \cite{Higson2002} and Alekseev--Finn-Sell \cite{Alekseev2018} as counterexamples to the Baum--Connes conjecture and examples of non-amenable groupoids with the weak containment property (see \cite{Willett2015}). 
We add to their list of exotic properties by showing that they also fail to be inner amenable:
\begin{alphthm}[{\autoref{thm:notinneramenable}}]\label{thmintro1}
	Let $\Gamma$ be a residually finite group with associated Higson--Lafforgue--Skandalis groupoid $G_{\mathrm{HLS}}$ and Alekseev--Finn-Sell groupoid $G_{\mathrm{AFS}}$. Then the following are equivalent. 
	\begin{enumerate}
		\item $\Gamma$ is amenable;
		\item $G_{\mathrm{HLS}}$ is inner amenable;
		\item $G_{\mathrm{AFS}}$ is inner amenable.
	\end{enumerate}
\end{alphthm}
By taking $\Gamma$ to be a non-amenable residually finite group such as the free group $\mathbb F_2$, this yields the first explicit example of a second countable Hausdorff \'etale groupoid which fails to be inner amenable in the sense of \cite{AnantharamanDelaroche}. 

Although \autoref{prob-inneramenable} is intimately related to \autoref{prob-transformation}, our counterexamples to the former are not automatically counterexamples to the latter, since the lack of inner amenability of a given \'etale groupoid only obstructs its realization as a partial transformation groupoid with \emph{clopen} domains. 
Nonetheless, we prove that the examples in \autoref{thmintro1} are, among many more, counterexamples to \autoref{prob-transformation}:

\begin{alphthm}
	The following groupoids cannot be realized as transformation groupoids associated to partial actions of discrete groups. 
	\begin{enumerate}
		\item Higson--Lafforgue--Skandalis and Alekseev--Finn-Sell groupoids associated to residually finite groups which are not locally finite (\autoref{cor-HLS});
		\item Deaconu--Renault groupoids associated to nontrivial self-coverings of connected spaces (\autoref{prop:no_pure_cocycle_connected});
		\item The coarse groupoid associated to the maximal uniformly locally finite structure on an infinite set (\autoref{exa-maximalcoarse}).
	\end{enumerate}
\end{alphthm}

For coarse groupoids, we in fact completely characterize when they arise as transformation groupoids in terms of coarse injections and embeddings into groups; this point of view was suggested to us in discussions with Exel \cite{Exelprivate}:
\begin{alphthm}[\autoref{thm:coarse}]
	Let $G(X)$ be the coarse groupoid associated to a uniformly locally finite coarse space. The following are equivalent:
	\begin{enumerate}
		\item $G(X)$ is isomorphic to a transformation groupoid by a partial group action;
		\item $X$ admits an injective coarse map into a discrete group.
	\end{enumerate}
	Moreover, the partial group action can be chosen to have clopen domains if and only if $X$ coarsely \emph{embeds} into a group. 
\end{alphthm}

While we do present an example of a coarse space that does not embed into a group, we do not know whether every uniformly locally finite \emph{metric} space coarsely embeds (or even injects) into a group. 

We complement our negative results with two positive results providing transformation groupoid models for a large class of well-studied $C^*$-algebras. 
Our first positive result is an adaptation of a theorem by Wu \cite{Wu2025} which presents every \emph{stable} UCT Kirchberg algebra as a crossed product by a (global) action of a discrete group $\Gamma$ on a locally compact space $X$.

\begin{alphthm}[\autoref{thm-kirchberg}]\label{introthm-kirchberg}
	Every \emph{unital} Kirchberg algebra in the UCT-class can be realized as the crossed product $C(Y)\rtimes_r \Gamma$ associated to a partial action of a countable discrete group $\Gamma$ on the Cantor set $Y$. 
\end{alphthm}
The proof of \autoref{introthm-kirchberg} is achieved by considering Wu's global actions $\Gamma\curvearrowright X$ and realizing every class in $K_0(C_0(X)\rtimes_r \Gamma)$ as the characteristic function $[1_Y]$ of a compact open subset $Y\subset X$. 
This strategy is similar to how \cite{CFH-Kirchberg} extract ample groupoid models for unital UCT Kirchberg algebras from the stable case treated in \cite{Spielberg2007}.
In order to do so, we identify a set of generators of $K_0(C_0(X)\rtimes_r \Gamma)$ from \cite{Brownlowe2025} and establish a dynamical comparison property in the sense of \cite{Kerr2020,Ma2022} allowing us to represent sums and differences of the generators by compact open subsets.

Our second positive result generalizes a result of Li \cite{Li2017} on graph groupoids and shows that path groupoids associated to many higher-rank graphs in the sense of \cite{Kumjian2000} do arise as transformation groupoids.

\begin{alphthm}[\autoref{prop-k-graph}]
Let $\Lambda$ be a higher-rank graph which embeds into its fundamental groupoid in the sense of \cite{Pask2004}. Then the associated path groupoid $G_\Lambda$ is isomorphic to a transformation groupoid arising from a partial action of a discrete group.
\end{alphthm}

Our approach is based on the following two key techniques. 
\begin{enumerate}
	\item The first key ingredient from \cite{deCastroKang2024} is a characterization of transformation groupoids in terms of \emph{pure cocycles} (see \autoref{def-pure}) $c\colon G\to \Gamma$ of the given groupoid $G$ into the given group $\Gamma$. The properties of $c$ (such as being locally constant and non-trivial) are quite rigid and allow us to obstruct the existence of transformation groupoid models for many Deaconu--Renault groupoids. In the case of coarse groupoids, this approach moreover gives a recipe to construct maps from the coarse space into a group out of a transformation groupoid model and vice versa.
	\item The second main technique is the abstraction of a structural property inherent to transformation groupoids which we call \emph{property $\Delta$} (see \autoref{def:Delta}). Property $\Delta$ is inspired by a stronger condition appearing in \cite[Lemma~5.16]{AnantharamanDelaroche} which aims to formalize the properties of the `diagonal' subgroupoid $\Gamma\ltimes (X\times X)\subset (\Gamma\ltimes X)\times (\Gamma\ltimes X)$ for a partial action $\Gamma\curvearrowright X$. 
	Our main observation is that property $\Delta$ can be obstructed whenever `infinite pieces of a groupoid can be uniformly approximated by compact pieces'. This core property of Higson--Lafforgue--Skandalis and Alekseev--Finn-Sell groupoids is exploited in \autoref{thm:HLS action-no-Delta-uniform} in order to show that they do not satisfy property $\Delta$ and therefore do not arise from partial actions. 
\end{enumerate}

\subsection*{Structure of the paper}

In \autoref{sec:property-delta} we introduce property~$\Delta$
and prove that all transformation groupoids arising from partial
actions satisfy it.
In \autoref{sec:HLS} we establish failure of property~$\Delta$
for HLS and AFS groupoids and deduce that they are not transformation
groupoids.
In \autoref{sec:dr-groupoids} we discuss Deaconu--Renault groupoids,
showing in particular that property~$\Delta$ is not sufficient
for being a transformation groupoid. We obtain our positive results concerning higher rank graphs here.
In \autoref{sec:coarse-groupoids} we characterize when coarse
groupoids arise from partial actions in terms of coarse embeddings into
groups.
In \autoref{sec:inner-amen} we study inner amenability and
prove that for HLS and AFS groupoids it is equivalent to amenability of
the underlying group.
Finally, in \autoref{sec:kirchberg} we realize all unital UCT Kirchberg algebras via partial actions of discrete groups.

\subsection*{Acknowledgements}
This work was supported by the Deutsche Forschungsgemeinschaft (DFG, German Research Foundation) under Germany's Excellence Strategy EXC 2044/2 - 3 90685587, Mathematics M\"unster: Dynamics--Geometry--Structure, and by the SFB 1442 of the DFG. The first author was supported by CNPq and the Humboldt Foundation.
We thank Gilles de Castro, Ruy Exel, Tim de Laat and Rufus Willett for stimulating discussions. We also thank Claire Anantharaman--Delaroche for sharing an early draft of \cite{AnantharamanDelaroche} with us and for her careful proofreading which helped us fix two errors in the first version of this manuscript.

\section{Transformation groupoids, pure cocycles and Property \texorpdfstring{$\Delta$}{Delta}}\label{sec:property-delta}
In this section we recall the definition of Hausdorff \'etale groupoids and develop general tools that allow us to characterize the subclass of transformation groupoids associated to partial actions among Hausdorff \'etale groupoids.
Recall that a \emph{groupoid} $G$ is a small category where all morphisms are invertible. 
Equivalently, a groupoid consists of a set $G$ of morphisms (called \emph{arrows}), a subset $G^{(0)}\subset G$ of objects\footnote{We identify every object with its identity morphism.} (called \emph{units}), maps $r,s\colon G\to G^{(0)}$ (called \emph{range} and \emph{source}), a map 
\[\cdot \colon G^{(2)}\coloneqq \{(g,h)\in G\times G\colon s(g)=r(h)\}\to G\]
(called \emph{multiplication}) and a necessarily unique map $(-)^{-1}\colon G\to G$ (called \emph{inversion}) satisfying a list of natural axioms (c.f.~\cite{Renault1980}).
We emphasize that the composition is written from right to left so that $gh$ is defined whenever $s(g)=r(h)$. Any group $\Gamma$ will be identified with the groupoid which has one unit and $\Gamma$ as endomorphisms. 

A \emph{topological groupoid} is a groupoid $G$ equipped with a topology such that all the above maps are continuous. 
A locally compact Hausdorff (topological) groupoid $G$ is called \emph{\'etale} if the range map $r\colon G\to G^{(0)}$ is a local homeomorphism. 
A subset $U\subset G$ is called a \emph{bisection} if $s|_U\colon U \to s(U)$ and $r|_U\colon U\to r(U)$ are homeomorphisms. 
Note that a locally compact Hausdorff groupoid is \'etale if and only if its topology has a basis of open bisections. 
For $u\in G^{(0)}$ we write $G_u\coloneqq\{g\in G\mid s(g)=u\}$ and
$G^u\coloneqq\{g\in G\mid r(g)=u\}$, and we write $G_u^u\coloneqq G_u\cap G^u$ for the isotropy group at $u$.
For a subset $K\subset G^{(0)}$, we denote by
\begin{equation}\label{eq-reduction}
	G(K)\coloneqq \{g\in G\mid s(g),r(g)\in K\}
\end{equation}
the \emph{reduction of $G$ to $K$.}
We endow the Cartesian product $G\times G$ with a groupoid structure such that $(G\times G)^{(0)}=G^{(0)}\times G^{(0)}$ and
\[
s(g,h)=(s(g),s(h)),\qquad r(g,h)=(r(g),r(h)),\qquad (g,h)^{-1}=(g^{-1},h^{-1}),
\]
for all $g,h\in G$. 

The following definition provides a vast source of examples of Hausdorff \'etale groupoids. We recommend \cite{Exel2017} for a comprehensive introduction to the partial actions in $C^*$-algebra theory.
\begin{defi}[\cite{Exel1994}]
	Let $\Gamma$ be a discrete group and let $X$ be a locally compact Hausdorff space. A \emph{partial action} $\Gamma\curvearrowright X$ consists of a collection of open subsets $\{D_\gamma\subset X\}_{\gamma\in \Gamma}$ and a collection of homeomorphisms $\{\gamma\cdot -\colon D_{\gamma^{-1}}\xrightarrow{\cong} D_\gamma\}_{\gamma\in \Gamma}$ such that 
	\begin{enumerate}
		\item $D_1=X$ and $1\in \Gamma$ acts as the identity;
		\item $\gamma^{-1}(D_\gamma\cap D_{\eta^{-1}})\subset D_{(\eta\gamma)^{-1}}$ for all $\gamma,\eta\in \Gamma$;
		\item ${(\eta\gamma)}(x)=\eta(\gamma(x))$ for all $x\in {\gamma}^{-1}(D_\gamma\cap D_{\eta^{-1}})$.
	\end{enumerate}
	The \emph{transformation groupoid} associated to the partial action is given by 
	\[\Gamma\ltimes X=\bigcup_{\gamma \in \Gamma}\{\gamma\}\times D_{\gamma^{-1}} \subset\Gamma \times X\]
	with unit space $(\Gamma\ltimes X)^{(0)}=\{1\}\times X\cong X$ and source, range and multiplication given by 
	 \[s(\gamma,x)=x,\quad  r(\gamma,x)=\gamma x,\quad (\eta,\gamma x)(\gamma,x)=(\eta\gamma,x)\]
	for $\gamma\in \Gamma,x\in D_{\gamma^{-1}}$ and $\eta\in \Gamma$ satisfying $\gamma x\in D_{\eta^{-1}}$. 
	The partial action is said to have \emph{closed graphs} if the sets 
	\[\{(x,\gamma x)\mid x\in D_{\gamma^{-1}}\}\subset X\times X,\quad \gamma\in \Gamma\] 
	are closed. 
\end{defi}
	Note that a partial action $\Gamma\curvearrowright X$ on a \emph{compact} space has closed graphs if and only if the domains $D_\gamma\subset X,\quad \gamma\in \Gamma$ are closed.

\begin{rem}
Partial transformation groupoids can also be defined for continuous actions of
topological groups in exactly the same way (see \cite{Abadie2004}). 
However, such transformation groupoids are \'etale if and only if the acting group is discrete.

Since all groupoids considered in this paper are \'etale, it follows that allowing
non-discrete topological groups does not lead to any additional examples. In
particular, whenever we show that a given étale groupoid cannot be realized as a
transformation groupoid of a partial action of a discrete group, it also cannot
arise from a partial action of any topological group.
\end{rem}

For a partial action $\Gamma\curvearrowright X$ as above, the associated transformation groupoid $\Gamma\ltimes X$ comes with a continuous groupoid homomorphism $c\colon \Gamma\ltimes X\to \Gamma$, $c(\gamma,x)=\gamma$ which completely recovers the partial action by a theorem of Castro--Kang \cite{deCastroKang2024}. Moreover, by an observation of Anantharaman-Delaroche, local properness of the cocycle is encoded by closedness of the graphs of the partial action.
is encoded by local properness of $c$. 
Recall from \cite[Definition~1.29]{AnantharamanDelaroche} that a continuous groupoid homomorphism $c\colon G\to \Gamma$ is called \emph{locally proper} if for every compact set $K\subset G^{(0)}$, the map 
\[r\times c\times s\colon G\to G^{(0)}\times \Gamma\times G^{(0)},\qquad  g\mapsto (r(g),c(g),s(g))\] 
is proper.
If $G^{(0)}$ is compact, local properness is equivalent to properness of $c$. %
\begin{theorem}[{\cite[Theorem 3.9]{deCastroKang2024}, \cite[Example 3.30(d)]{AnantharamanDelaroche}}]\label{thm:deCastroKang} %
	Let $G$ be a locally compact Hausdorff \'etale groupoid and let $\Gamma$ be a discrete group. Then the following are equivalent:
	\begin{enumerate}
		\item There is a partial action $\Gamma\curvearrowright G^{(0)}$ such that $G\cong \Gamma\ltimes G^{(0)}$;
		\item\label{item:pure} There is a continuous groupoid homomorphism $c\colon G\to \Gamma$ such that 
		\[\ker(c)\coloneqq \{g\in G\mid c(g)=1\}=G^{(0)}.\]
	\end{enumerate}
	The map $c\colon G\to \Gamma$ in item~\ref{item:pure} is locally proper if and only if the partial action $\Gamma\curvearrowright G^{(0)}$ has closed graphs.
	Moreover, then the preimages $c^{-1}(\gamma)\subset G$ for $\gamma\in \Gamma$ are clopen bisections.
\end{theorem}
\begin{proof}
	The characterization of local properness of $c$ is due to \cite[Example 3.30(d)]{AnantharamanDelaroche}.  
	We only prove the last statement since it is not explicitly stated in \cite[Theorem 3.9]{deCastroKang2024}.  

	Let $\gamma\in \Gamma$ and let $g,h\in c^{-1}(\gamma)$ with $s(g)=s(h)$.
	Then, we have $c(gh^{-1})=c(g)c(h)^{-1}=\gamma\gamma^{-1}=1$ which implies $gh^{-1}\in \ker(c)=G^{(0)}$ and therefore $g=h$.
	This proves injectivity of $s|_{c^{-1}(\gamma)}$, and analogously of $r|_{c^{-1}(\gamma)}$.
\end{proof}
\begin{defi}\label{def-pure}
	A map $c\colon G\to \Gamma$ satisfying the condition in \autoref{thm:deCastroKang} \ref{item:pure} is called a \emph{pure cocycle}. 
\end{defi}

The following definition is inspired by \cite[Lemma~5.16]{AnantharamanDelaroche} where what we call \emph{property $\bar\Delta$} is used as a sufficient condition for Anantharaman-Delaroche's \emph{inner amenability} of a Hausdorff 
\'etale groupoid. 
Our definition of \emph{property $\Delta$} axiomatizes the properties of the subgroupoid $\Gamma\ltimes (X\times X)\subset (\Gamma\ltimes X)\times (\Gamma\ltimes X)$ for a partial action $\Gamma\curvearrowright X$ and is designed as an obstruction for a given groupoid against being isomorphic to a transformation groupoid. 
Compared to the theorem of Castro and Kang, this condition has the advantage of not including an auxiliary discrete group in its definition. 
We say that a subset $A\subset G$ of a Hausdorff \'etale groupoid $G$ has \emph{uniformly finite fibers}, if 
\begin{equation}\label{eq:defunifinite}
	\#_{\mathrm{fib}}(A)\coloneqq \sup_{x\in G^{(0)}}(|A\cap G_x|+|A\cap G^x|)<\infty.
\end{equation}
\begin{defi}[Property $\Delta$]\label{def:Delta}
A locally compact Hausdorff \'etale groupoid $G$ has \emph{property $\Delta$}
if there is an open subgroupoid $H\subset G\times G$ containing the diagonal
$\Delta_G\coloneqq \{(g,g)\mid g\in G\}$ such that for every compact subset
$K\subset G^{(0)}$ and every compact subset $C\subset G(K)$, the sets
\begin{equation}\label{eq:uniformlyfinite}
H\cap (G(K)\times C)
\qquad\text{and}\qquad
H\cap (C\times G(K))\subset G\times G
\end{equation}
have uniformly finite fibers (with respect to $G\times G$).
If $H$ is additionally closed and the sets in \eqref{eq:uniformlyfinite} are compact, then $G$ is said to have \emph{property $\bar \Delta$}.
\end{defi}

\begin{rem}
	We do not know if property $\Delta$ implies inner amenability. 
\end{rem}

The property $\Delta$ introduced above is designed so that all partial
transformation groupoids satisfy it. A stronger variant, property
$\bar\Delta$, corresponds to the case where the graphs of the partial action are closed.

\begin{lemma}[Partial transformation groupoids have property $\Delta$]\label{lem:partial-actions-have-Delta}
Let $\Gamma\ltimes X$ be the transformation groupoid associated to a partial action
of a discrete group $\Gamma$ on a locally compact Hausdorff space $X$.
Then $\Gamma\ltimes X$ has property $\Delta$. If $\Gamma\curvearrowright X$ has closed graphs, then $\Gamma\ltimes X$ moreover has property $\bar\Delta$.
\end{lemma}

\begin{proof}
Define
\[
H\coloneqq \{((\gamma,x),(\gamma,y))\mid \gamma\in \Gamma,\ x,y\in D_{\gamma^{-1}}\}
\subset (\Gamma\ltimes X)\times (\Gamma\ltimes X).
\]
Since $\Gamma$ is discrete and the domains are open, $H$ is open.
It is immediate that $H$ is a subgroupoid and contains the diagonal.

Let $K\subset X$ be compact and let $C\subset (\Gamma\ltimes X)(K)$ be compact.
The projection $\pr_\Gamma\colon \Gamma\ltimes X\to\Gamma$, $(\gamma,x)\mapsto\gamma$
is continuous and $\Gamma$ is discrete, hence $\pr_\Gamma(C)$ is finite.
Thus there exists a finite subset $F\subset \Gamma$ such that $C\subset \bigcup_{\gamma\in F} C_\gamma$,
where each $C_\gamma\coloneqq C\cap (\{\gamma\}\times D_{\gamma^{-1}})$ is compact since $\{\gamma\}\times D_{\gamma^{-1}}\subset \Gamma\ltimes X$ is closed.
We have
\begin{equation}\label{eq:HcapKC}
H\cap\bigl((\Gamma\ltimes X)(K)\times C\bigr)
\subset \bigcup_{\gamma\in F} \Bigl(\bigl(\{\gamma\}\times(K\cap D_{\gamma^{-1}})\bigr)\times C_\gamma\Bigr).
\end{equation}
For each fixed unit of \( ( \Gamma\ltimes X)\times (\Gamma\ltimes X)\), each summand contributes at most one element to the source fibre and at most one element to the range fibre. Hence
\[
\#_\mathrm{fib}\bigl(H\cap((\Gamma\ltimes X)(K)\times C)\bigr)\le 2|F|,
\]
where ${\#_{\mathrm{fib}}}$ is defined in \eqref{eq:defunifinite} (but with respect to $G\times G$). The same argument applies to $H\cap(C\times (\Gamma\ltimes X)(K))$. This shows that $\Gamma\ltimes X$ has property $\Delta$. 

Now assume further that $\Gamma\curvearrowright X$ has closed graphs. 
Then the cocycle 
\[c\colon \Gamma\ltimes X\to \Gamma,\quad c(\gamma,x)=\gamma\]
is locally proper by \autoref{thm:deCastroKang}. 
It follows from the proof of \cite[Proposition~5.17]{AnantharamanDelaroche} that $\Gamma\ltimes X$ has property $\bar \Delta$.
\end{proof}

\section{Higson--Lafforgue--Skandalis and Alekseev--Finn-Sell groupoids}\label{sec:HLS} 

Our first class of Hausdorff \'etale groupoids which are not transformation groupoids by partial actions goes back to Higson--Lafforgue--Skandalis \cite{Higson2002} and Alekseev--Finn--Sell \cite{Alekseev2018}. 
Higson--Lafforgue--Skandalis groupoids (\emph{HLS groupoids}) were introduced in \cite{Higson2002} as counterexamples to the Baum--Connes conjecture. The main feature of these groupoids for disproving the Baum--Connes conjecture are their non-exactness properties which persist at the level of $K$-theory. These non-exactness properties were later used by Willett \cite{Willett2015} to provide an example of a non-amenable \'etale groupoid whose full and reduced $C^*$-algebra agree. Later, Alekseev--Finn-Sell groupoids (\emph{AFS groupoids}) were constructed as a principal variant of HLS groupoids with similar properties. 

In this section, we add to the list of exotic properties of HLS and AFS groupoids by proving that they lack property $\Delta$ (\autoref{def:Delta}) and therefore do not arise as transformation groupoids from partial group actions.
The argument relies on the existence of large families of elements which can be approximated by compact pieces and which remain uniformly close to the diagonal in $G\times G$ under translations by the acting group.
This phenomenon occurs near the
fiber at infinity and reflects the equicontinuity of the translation action of a group on its profinite completion. 
We will moreover exhibit HLS and AFS groupoids as the first examples of non inner amenable groupoids in \autoref{sec:inner-amen}.

We recall the constructions from \cite{Higson2002,Alekseev2018} below.

\begin{defi}
	An \emph{approximated group} is a pair $(\Gamma, (N_k)_{k\in \N})$ where $\Gamma$ is a discrete group and 
	\[
\Gamma \supset N_1 \supset N_2 \supset \cdots
\]
is a decreasing sequence of finite-index normal subgroups such that $\bigcap_{k\in\N} N_k=\{1\}$.
\end{defi}

Let $(\Gamma,(N_k)_{k\in \N})$ be an approximated group. 
Denote by $\N\cup\{\infty\}$ the one-point compactification of the nonnegative integers $\N=\{0,1,2,\dotsc\}$.
For $k\in\N\cup\{\infty\}$, we write $\Gamma_k\coloneqq \Gamma/N_k$, where $N_\infty\coloneqq \{1\}$ and denote by
$\pi_k\colon\Gamma\to \Gamma_k$ the quotient map (with $\pi_\infty=\mathrm{id}_\Gamma$).
We define a topology on 
\[
G_{\mathrm{HLS}}(\Gamma,(N_k)_{k\in \N}) \coloneqq \coprod_{k\in \N\cup\{\infty\}} \{k\}\times \Gamma_k,
\]
by declaring the singletons $\{(k,\gamma)\}$ for $k\in \N$ and $\gamma\in \Gamma_k$ as well as the sets
\begin{equation}\label{Nngamma}
N(n,\gamma)\coloneqq \{(k,\pi_k(\gamma))\mid\ k\in\N\cup \{\infty\},\ k\ge n\},
\end{equation}
for $n\in \N$ and $\gamma\in \Gamma$ to be open (which incidentally forces them to be compact). 
Then $G_{\mathrm{HLS}}(\Gamma,(N_k)_{k\in \N})$ is a locally compact Hausdorff groupoid with $r,s\colon G_{\mathrm{HLS}}(\Gamma,(N_k)_{k\in \N})\to \N\cup\{\infty\}$ given by the projection and multiplication induced by the multiplication of $\Gamma$. 
\begin{defi}
	We call $G=G_{\mathrm{HLS}}(\Gamma,(N_k)_{k\in \N})$ the Higson--Lafforgue--Skandalis groupoid (\emph{HLS groupoid}) associated to the approximated group $(\Gamma,(N_k)_{k\in \N})$.
\end{defi}

\begin{defi}
	Let $G$ be a Hausdorff \'etale groupoid and let $Y$ be a locally compact Hausdorff space. 
	An \emph{action} $G\curvearrowright Y$ consists of a continuous map $\rho\colon Y\to G^{(0)}$ (the \emph{anchor map}) and a continuous map 
	\[-\cdot- \colon G\times_{s,\rho}Y\coloneqq \{(g,y)\in G\times Y\mid s(g)=\rho(y)\}\to Y\]
	(the \emph{action map}) such that 
	\begin{enumerate}
		\item $\rho(g\cdot y)=r(g)$ for all $(g,y)\in G\times_{s,\rho}Y$;
		\item $g\cdot (h\cdot y)=(gh)\cdot y$ for all $(g,h)\in G^{(2)}$ and $y\in Y$ with $s(h)=\rho(y)$.
	\end{enumerate}
	Given an action $G\curvearrowright Y$, the associated \emph{transformation groupoid} is defined as 
	$G\ltimes Y\coloneqq G\times_{s,\rho}Y$
	with unit space $Y\cong G^{(0)}\times_{s,\rho}Y$ and 
	\[s(g,y)\coloneqq y,\quad r(g,y)\coloneqq g\cdot y,\quad (h,gy)(g,y)\coloneqq (hg,y)\]
	for all $(g,y)\in G\ltimes Y$ and $h\in G$ with $s(h)=r(g)$. 
\end{defi}

Let $(\Gamma,(N_k)_{k\in \N})$ be an approximated group. 
Denote by $\widehat \Gamma_\infty\coloneqq \varprojlim_{k\in \N} \Gamma_k$ the profinite completion of $\Gamma$ with respect to the family $(N_k)_{k\in \N}$\footnote{Note that the definition of $\widehat \Gamma_\infty$ depends on the choice of approximating sequence $(N_k)_{k\in \N}$.}. Concretely, $\widehat \Gamma_\infty$ is given by the closed image of the diagonal map 
\[\Gamma\hookrightarrow \prod_{k\in \N}\Gamma_k,\quad \gamma\mapsto (\pi_k(\gamma)_{k\in \N}).\] 
For $k\in \N$, we define $\widehat \Gamma_k\coloneqq \Gamma_k$ and denote by $\widehat\pi_{k,n}\colon \widehat \Gamma_n\to \widehat \Gamma_k$ the canonical quotient map for $k\leq n\leq \infty$. 
We define a topology on 
\begin{equation}\label{alekseevfinnsellspace}
	\widehat X\coloneqq \coprod_{k\in \N\cup\{\infty\}} \{k\}\times \widehat \Gamma_k
\end{equation}
by declaring as basic open sets the singletons $\{(k,\gamma)\}$ for $k\in \N$ and $\gamma\in \widehat \Gamma_k$ as well as the \emph{shadows}
\begin{equation}\label{shadow}
	\mathrm{Sh}(\gamma)\coloneqq \bigcup_{k\leq n\leq \infty}\{k\}\times \widehat \pi_{k,n}^{-1}(\gamma),
\end{equation}
for $k\in \N$ and $\gamma\in \widehat \Gamma_k$.
We define an action $G_{\mathrm{HLS}}(\Gamma,(N_k)_{k\in \N})\curvearrowright \widehat X$ by defining the anchor map $\rho \colon \widehat X\to \N\cup \{\infty\}$ to be the projection and the action map
\[G_{\mathrm{HLS}}(\Gamma,(N_k)_{k\in \N})\times_{s,\rho}\widehat X\to \widehat X\]
induced by the left multiplication maps $\Gamma_k\times \widehat \Gamma_k\to \widehat \Gamma_k$ for each $k\in \N\cup \{\infty\}$. 

\begin{defi}
	The associated transformation groupoid 
	\[ G_{\mathrm{AFS}}(\Gamma,(N_k)_{k\in \N})\coloneqq G_{\mathrm{HLS}}(\Gamma,(N_k)_{k\in \N})\ltimes \widehat X\]
	is called the Alekseev--Finn-Sell groupoid (\emph{AFS groupoid}) associated to $(\Gamma,(N_k)_{k\in \N})$.
\end{defi}

The failure of property $\Delta$ in the HLS and AFS constructions
comes from the presence of large ``equicontinuous'' pieces of the
groupoid near the fiber at infinity.
More precisely, one can find a boundary point $x_0\in \widehat X_\infty$ and a finite set $S\subseteq \Gamma$ such that the subgroup generated by $S$ act equicontinuously at $x_0$. 
After approximating $x_0$ by points in one of the finite fibers, this produces
infinite fibers in any open subgroupoid containing the diagonal, violating the uniform finiteness condition in the definition of property $\Delta$.

For this we will need the following key property for the action $\Gamma\curvearrowright \widehat X$ defining the AFS-groupoid.

\begin{defi}%
	An action $\Gamma\curvearrowright Y$ of a discrete group on a compact space is called \emph{equicontinuous at $y_0\in Y$} if for every open set $U\subset Y\times Y$ containing the diagonal $\Delta_Y$, there is an open neighbourhood $V$ of $y_0$ such that 
	\[\bigcup_{\gamma\in \Gamma}\gamma(\{y_0\}\times V)\subset U.\]
\end{defi}

\begin{prop}\label{prop:AFScontinuous}
	Let $(\Gamma,(N_k)_{k\in \N})$ be an approximated group. Consider the space 
	\[\widehat X\coloneqq \coprod_{k\in \N\cup\{\infty\}} \{k\}\times \widehat \Gamma_k\]
	as in \eqref{alekseevfinnsellspace}. Then the action $\Gamma\curvearrowright \widehat X$ given by $\gamma\cdot (k,\eta)\coloneqq (k,\pi_k(\gamma)\eta)$ for all $k\in \N\cup \{\infty\}, \gamma\in \Gamma,\eta\in \widehat\Gamma_k$ is equicontinuous at every point.%
\end{prop}
\begin{proof}
	Fix $x_0\in \widehat X$ and an open neighbourhood $U$ of the diagonal $\Delta_{\widehat X}\subset \widehat X\times \widehat X$. 
	For each $k\in \N\cup \{\infty\}$, we write $\widehat X_k\coloneqq \rho^{-1}(k)=\{k\}\times \widehat{\Gamma}_k\subset \widehat X$. 
	If $x_0\in \widehat X_k$ for some $k\in \N$, then the set $V=\{x_0\}$ is an open neighbourhood of $x_0$ so that $\bigcup_{\gamma\in \Gamma}\gamma (\{x_0\}\times V)\subset \Delta_{\widehat X}\subset U$.

	Now assume that $x_0\in \widehat X_\infty$. 
	By compactness, we may without loss of generality assume that there exist $k_i\in \N$ and $\gamma_i\in \Gamma_{k_i}$ for $i=1,\dotsc,m$ such that 
	\[U= \Delta_{\widehat X\setminus \widehat X_\infty}\cup \mathrm{Sh}(\gamma_1)^{\times 2}\cup\dotsb\cup \mathrm{Sh}(\gamma_m)^{\times 2}\subset \widehat X \times \widehat X,\]
	where $\mathrm{Sh}(\gamma_i)$ is as in \eqref{shadow}. 
	Define $k\coloneqq \max\{k_1,\dotsc,k_m\}$ and $V\coloneqq \mathrm{Sh}(\widehat\pi_{k,\infty}(x_0))$. 
	We claim that 
	\[\bigcup_{\gamma\in \Gamma}\gamma(\{x_0\}\times V)\subset U.\]
	To see this, let $\gamma\in \Gamma$. Since $U$ contains the diagonal $\Delta_{\widehat X_\infty}$, there is $i=1,\dotsc,m$ with $\gamma x_0\in \mathrm{Sh}(\gamma_i)$, which by definition means $\widehat \pi_{k_i,\infty}(\gamma x_0)=\gamma_i$. 
	Now let $\eta\in \gamma V=\mathrm{Sh}(\widehat \pi_{k,\infty}(\gamma x_0))$, and let $n\geq k$ be such that $\eta\in \widehat X_n$. Then, we have 
	\[\widehat \pi_{k_i,n}(\eta)=\widehat \pi_{k_i,k}(\widehat \pi_{k,n}(\eta))=\widehat \pi_{k_i,k}(\widehat \pi_{k,\infty}(\gamma x_0))=\widehat \pi_{k_i,\infty}(\gamma x_0)=\gamma_i,\]
	and thus $\eta\in \mathrm{Sh}(\gamma_i)$. 
	This proves that $\gamma (\{x_0\}\times V)\subset U$. 
\end{proof}

The following theorem and corollary are the main results of this section.

\begin{theorem}\label{thm:HLS action-no-Delta-uniform}
Let $(\Gamma,(N_k)_{k\in \N})$ be an approximated group with associated HLS groupoid $G=G_{\mathrm{HLS}}(\Gamma,(N_k)_{k\in \N})$. 
Let $G\curvearrowright Y$ be an action
on a compact Hausdorff space with open anchor map $\rho\colon Y\to\N\cup\{\infty\}$.
Assume that there is a finitely generated infinite subgroup $\Lambda\subset \Gamma$ such that the induced action $\Lambda\curvearrowright Y$ is equicontinuous at some point $y_0\in Y_\infty\coloneqq \rho^{-1}(\infty)$. 
Then $G\ltimes Y$ does not have property $\Delta$.
\end{theorem}
Recall that a discrete group is called \emph{locally finite} if every finitely generated subgroup is finite.
Together with \autoref{prop:AFScontinuous}, \autoref{thm:HLS action-no-Delta-uniform} implies the following.

\begin{cor}\label{cor-HLS}
Let $(\Gamma,(N_k)_{k\in \N})$ be an approximated group such that $\Gamma$ is not locally finite. 
Then $G_{\mathrm{HLS}}(\Gamma,(N_k)_{k\in \N})$ and $G_{\mathrm{AFS}}(\Gamma,(N_k)_{k\in \N})$ do not have property $\Delta$. In particular, they do not arise as transformation groupoids associated to partial actions of discrete groups. 
\end{cor}

The above applies to many amenable groups, even to abelian ones. 
For instance, if $\Gamma=\Z$, then both $G_{\mathrm{HLS}}$ and $G_{\mathrm{AFS}}$ are amenable groupoids. 
Moreover, when $\Gamma$ is abelian the groupoid $G_{\mathrm{HLS}}$ is an abelian group bundle, and hence its $C^*$-algebra is commutative.

This shows in particular that a groupoid $C^*$-algebra may admit a realization as a crossed product by a group action on a commutative $C^*$-algebra, even though the underlying groupoid itself does not arise as a transformation groupoid of any partial action of a discrete group.

\begin{proof}[Proof of \autoref{thm:HLS action-no-Delta-uniform}]
	Throughout the proof of \autoref{thm:HLS action-no-Delta-uniform}, we use the notation $A^{\times 2}\coloneqq A\times A$ to abbreviate the product of a set $A$ with itself. We moreover abbreviate the fibers of the anchor map $\rho$ as $Y_k\coloneqq \rho^{-1}(k)$ for all $k \in \N\cup\{\infty\}$. 
	Let $H\subset (G\ltimes Y)^{\times 2}$ be an open subgroupoid containing the diagonal $\Delta_{G\ltimes Y}$. We show that $H$ does not satisfy the conditions of \autoref{def:Delta}. 

	Let $S=S\cup S^{-1}\cup \{1\}$ be a finite generating set for $\Lambda$. 
	Since $H$ is open and contains the diagonal, we may find for every $z\in Y$ an open neighbourhood $z\in U\subset Y$ and $n\in \N$ such that
	\[\Big\{(k,\pi_k(\gamma),y)\,\Big\vert\, y\in U\cap Y_k,\, n\leq k \leq \infty\Big\}^{\times 2}\subset H,\]
	for all $\gamma\in S$. 
	Using compactness of $Y_\infty$, we may thus find an open cover ${Y_\infty\subset U_0\cup\dotsb\cup U_N\subset Y}$ and $n_0,\dotsc,n_N\in \N$ such that 
	\begin{equation}\label{eq:U_i}
		\Big\{(k,\pi_k(\gamma),y)\,\Big\vert\, y\in U_i\cap Y_k,\, n_i\leq k \leq \infty\Big\}^{\times 2}\subset H,
	\end{equation}
	for all $\gamma\in S$ and $i=0,\dotsc,N$. 
	Without loss of generality, we may assume that $y_0\in U_0$. 
	Note that 
	\[U\coloneqq \left(\bigcup_{n\in \N}\Delta_{Y_n}\right)\cup\left(\bigcup_{i=0}^NU_i\times U_i\right)\subset Y\times Y\]
	is an open neighbourhood of the diagonal $\Delta_Y$. By the equicontinuity assumption, we may find an open neighbourhood $y_0\in V\subset U_0$ such that 
	\begin{equation}\label{eq:UV}
		\bigcup_{\gamma\in \Lambda}\gamma(\{y_0\}\times V)\subset U.
	\end{equation}
	Since $\rho$ is open, we may find a point $y_n\in V\cap Y_n$ with $\max\{n_0,\dotsc,n_N\}\leq n<\infty$.

	We claim that for every $\gamma\in \Lambda$, we have 
	\begin{equation}\label{eq:infinitediscrete}
		\big\{((n,\pi_n(\gamma),y_n),(\infty,\gamma,y_0))\,\big\vert\, \gamma \in \Lambda\big\}\subset H.
	\end{equation}
	To see this, write $\gamma=\gamma_1\cdot \dotsb \gamma_l$ for $\gamma_1,\dotsc,\gamma_l\in S$. 
	It follows from $y_0,y_n\in V\subset U_0$ together with \eqref{eq:UV} that for each $j=1,\dotsc,l$, there is an $i(j)=0,\dotsc,N$ such that $\gamma_{j+1} \dotsb \gamma_l y_0,\gamma_{j+1} \dotsb \gamma_l y_n\in U_{i(j)}$. 
	Using this, \eqref{eq:U_i} implies that 
	\[((n,\pi_n(\gamma_j),\gamma_{j+1} \dotsb \gamma_l y_n),(\infty,\gamma_j,\gamma_{j+1} \dotsb \gamma_ly_0))\in H,\]
	for all $j=1,\dotsc, l$. 
	Since $H$ is a subgroupoid, this proves \eqref{eq:infinitediscrete}.

	Define $C\coloneqq \Gamma_n\ltimes Y_n\subset G\ltimes Y$ which is a compact set. 
	By \eqref{eq:infinitediscrete}, the projection of the set $(C\times (G\ltimes Y))\cap H$ onto the second factor contains the infinite set $\{\infty\}\times \Lambda\times \{y_0\}\subset G\ltimes Y$. 
	In particular, $(C\times (G\ltimes Y))\cap H$ does not have uniformly finite fibers.
\end{proof}

The following proposition and example show that the assumption of $\Gamma$ containing a finitely generated infinite subgroup is not redundant. 
\begin{prop}
	Let $(\Gamma,(N_k)_{k\in \N})$ be an approximated group with associated HLS groupoid $G=G_{\mathrm{HLS}}(\Gamma,(N_k)_{k\in \N})$ such that $\Gamma$ is locally finite and countable. Let $G\curvearrowright Y$ be an action on a compact space. Then $G\ltimes Y$ has property $\bar\Delta$. 
\end{prop}
\begin{proof}%
	As above, we use the notation $A^{\times 2}\coloneqq A\times A$ to abbreviate the product of a set $A$ with itself. 
	Let 
	\[F_1\subset F_2\subset \dotsb\subset \Gamma\]
	be an increasing sequence of finite subgroups such that $F_\infty\coloneqq \bigcup_{n\in \N}F_n=\Gamma$.
	By possibly enlarging the subgroups, we can without loss of generality assume that 
	\begin{equation}\label{pinsurjective}
		\pi_n(F_n)=\Gamma_n,\quad \text{ for all } n\in \N,
	\end{equation}
	where $\pi_n\colon \Gamma\to \Gamma_n\coloneqq \Gamma/N_n$ denotes the quotient map.
	We define 
	\[H\coloneqq (G\ltimes Y)^{\times 2}\cap \left(\bigcup_{n\in \N}\bigcup_{\gamma\in F_n}(N(n,\gamma)\times Y)^{\times 2}\right),\]
	where $N(n,\gamma)$ is defined as in \eqref{Nngamma}.
	Then $H$ is open since each $N(n,\gamma)$ is open. 
	To prove that $H$ is closed, it suffices to prove that $H\cap K$ is closed for any compact subset $K\subset (G\ltimes Y)^{\times 2}$ by \cite[Proposition~1.7]{Strickland2009}. Without loss of generality,
	there is an integer $n\in \N$ such that $K=(G\ltimes Y)^{\times 2}\cap  \left(\bigcup_{\gamma\in F_n}N(1,\gamma)\right)^{\times 2}$. 
	Then we have 
	\[H\cap K=(G\ltimes Y)^{\times 2}\cap \bigcup_{k=1}^n\bigcup_{\gamma \in F_k}(N(k,\gamma)\times Y)^{\times 2},\]
	which is compact and thus closed. This proves that $H$ is closed. 

	We show that $H\subset (G\ltimes Y)^{\times 2}$ is a subgroupoid.
	Let ${(\gamma',\gamma y,\eta',\eta z),(\gamma,y,\eta,z)\in H}$ with 
	$r(\gamma',\eta')=r(\gamma,\eta)=(n,m)\in (\N\cup\{\infty\})^{\times 2}$. We may assume $m\geq n$.
	If $n\in \N$, then we must have $\gamma,\eta\in N(n,\xi)$ and $\gamma',\eta'\in N(n,\xi')$ for some $\xi,\xi'\in F_n\subset F_m$. 
	In particular, 
	\[(\gamma',\gamma y,\eta',\eta z),(\gamma,y,\eta,z)=(\gamma'\gamma,y,\eta'\eta,z)\in (N(n,\xi'\xi)\times Y)^{\times 2}.\]
	If $n=m=\infty$, there must be $k\in \N$ and $\gamma=\eta,\gamma'=\eta'\in F_k$. In particular,
	\[(\gamma',\gamma y,\eta',\eta z),(\gamma,y,\eta,z)=(\gamma'\gamma,y,\gamma'\gamma,z)\in (N(k,\gamma'\gamma)\times Y)^{\times 2},\]
	proving that $H$ is a subgroupoid.

	Moreover, $H$ contains the diagonal $\Delta_{G\ltimes Y}$ by \eqref{pinsurjective}.
	Let $C\subset G\ltimes Y$ be a compact subset. 
	Without loss of generality, we can assume  
	\[C\subset (G\ltimes Y)\cap \left(\bigcup_{\gamma\in F_n}N(1,\gamma)\times Y\right)\]
	for some $n\in \N$. Fix $m\geq n$ such that $\pi_m\colon \Gamma\to \Gamma_m$ is injective on $F_n\subset \Gamma$. 
	Then,
	\begin{align*}
		H\cap(C\times (G\ltimes Y))\subset&  (G\ltimes Y)^{\times 2}\cap\left(\bigcup_{k=1}^m\bigcup_{\gamma\in F_k}(N(k,\gamma)\times Y)^{\times 2}\dotsb \right.\\
		&\left.\dotsb\cup \bigcup_{\gamma\in F_n}\underbrace{\Big\{(\pi_i(\gamma),y,\pi_j(\gamma),z)\,\Big\vert\, i,j\geq m, y,z\in Y\Big\}}_{\cong ((\N\cup\{\infty\})\times Y)^{\times 2}}\right)
	\end{align*}
	is compact in $(G\ltimes Y)^{\times 2}$. 
	Analogously, $H\cap ((G\ltimes Y)\cap C)\subset (G\ltimes Y)^{\times 2}$ is compact. 
\end{proof}

\begin{exa}
Let $\Gamma=\bigoplus_{i=0}^\infty \Z/2\Z$ and $N_k=\bigoplus_{i=k}^\infty\Z/2\Z$ for all $k\in \N$. 
Then $G_{\mathrm{HLS}}(\Gamma,(N_k))$ is isomorphic to the transformation groupoid $\Gamma\ltimes (\N\cup\{\infty\})$ for the trivial partial action with domains 
\[D_\gamma =\{n\in \N\cup\{\infty\}\mid n> \max\{k\in \N \mid \gamma_k\not=0\}\},\quad \gamma\in \Gamma\setminus \{1\}.\]
\end{exa}

\section{Deaconu--Renault groupoids}\label{sec:dr-groupoids}
In this section, we use the criterion of Castro--Kang from \autoref{thm:deCastroKang} in terms of pure cocycles in order to prove that many Deaconu--Renault groupoids, including those associated to non-trivial self coverings of connected spaces, do not arise as transformation groupoids associated to partial actions of discrete groups. 
We illustrate that property $\Delta$ is not the only obstruction to being a transformation groupoid by proving $\bar\Delta$ (see \autoref{def:Delta}) for all Deaconu--Renault groupoids.
We moreover prove that path groupoids associated to many higher rank graphs do arise as transformation groupoids -- providing a sharp distinction between the connected and totally disconnected setting. 
We introduce Deaconu--Renault groupoids in the generality of $\N^k$-actions in order to treat higher rank graphs.
Deaconu--Renault groupoids were introduced in \cite{Renault1980,Deaconu1995,AnantharamanDelaroche1997}.

\begin{defi}
Let $k\in \N$ and let $X$ be a locally compact Hausdorff space. An \emph{$\N^k$-action} $\sigma\colon \N^k\curvearrowright X$ is a monoid homomorphism $\N^k\ni k\mapsto \sigma^k\in \mathrm{LHomeo}(X)$ from $\N^k$ into the monoid of local homeomorphisms $X\to X$. 
The \emph{Deaconu--Renault groupoid} $G_\sigma$ associated to $\sigma$ is given by 
\[G_\sigma\coloneqq \big \{(x,n,y)\in X\times \Z^k \times X\,\big\vert\,\exists m,l\in \N^k,\, \sigma^mx=\sigma^ly, \, n=m-l\big\},\]
with unit space $G_\sigma^{(0)}=X$ and 
\[r(x,n,y)= x,\qquad s(x,n,y)=y, \qquad (x,n,y)(y,m,z)=(x,n+m,z),\]
for all $(x,n,y),(y,m,z)\in G_\sigma$. 
\end{defi}
\begin{exa}
	The following Hausdorff--\'etale groupoids can be realized as Deaconu--Renault groupoids:
	\begin{enumerate}
		\item Transformation groupoids associated to (global) $\Z^k$-actions;
		\item Path groupoids associated to higher rank graphs \cite{Kumjian2000};
		\item Transformation groupoids of semi-saturated orthogonal partial actions of free groups \cite{Steinberg2026}.
	\end{enumerate}
\end{exa}
\begin{rem}\label{rem:k-graph}
	It is shown in \cite[Theorem 2.2]{Steinberg2026} that Deaconu--Renault groupoids associated to $\N$-actions on totally disconnected compact Hausdorff spaces are transformation groupoids of partial actions by free groups.\footnote{Note that in our definition of Deaconu--Renault groupoids, the local homeomorphisms are globally defined so that the extra assumption in \cite[Theorem 2.2]{Steinberg2026} is automatically satisfied. }  We do not know whether the same holds true for $\N^k$-actions. We thank Ruy Exel for bringing this question to our attention. 
	The following proposition provides some evidence for a positive answer. %
\end{rem}

\begin{prop}
	Let $k\in \N$ and let $\sigma\colon \N^k\curvearrowright X$ be an action on a locally compact Hausdorff space. 
	Then the associated Deaconu--Renault groupoid $G_\sigma$ has property $\bar\Delta$. 
\end{prop}
\begin{proof}
	We define 
	\[H\coloneqq \{ ((x,n,y),(x',n,y'))\mid (x,n,y),(x',n,y')\in G_\sigma\}\subset G_\sigma\times G_\sigma.\]
	It is easy to see that $H$ is a clopen subgroupoid. 
	Let $K\subset X$ and $C\subset G_\sigma(K)$ be compact subsets. 
	Then there is a finite subset $F\subset \Z^k$ such that $C\subset K\times F \times K$.
	In particular, $H\cap (G_\sigma(K)\times C)$ is contained in $K\times F \times K\times K \times F\times K$ which is compact. 
	Thus, $H\cap (G_\sigma(K)\times C)$ is compact. 
	Analogously, $H\cap (C\times G_\sigma(K))$ is compact. 
\end{proof}

More evidence that all Deaconu--Renault groupoids with totally disconnected unit space might arise from partial actions is given by path groupoids associated to higher rank graphs which embed into their fundamental groupoids \cite{Kumjian2000,Pask2004}. In this case, we can answer Exel's question affirmatively.  We recall the basic definitions from \cite{Kumjian2000}: 

A \emph{$k$-graph} for $k\geq 1$ is a small category $\Lambda$ together with a functor $d\colon \Lambda \to \N^k$, such that the following \emph{factorization property} holds: For all morphisms $\lambda\in \Lambda$ and factorizations $d(\lambda)=n+m$ with $n,m\in \N^k$, there are unique composable morphisms $\mu,\eta\in \Lambda$ with $n=d(\mu)$, $m=d(\eta)$ and $\lambda = \mu\eta$.
Here, we adopt the conventions for groupoids from \autoref{sec:property-delta} and view $\Lambda$ as a set of morphisms together with range, source and multiplication maps and $\N^k$ as a category with one object and $\N^k$ as endomorphisms. We again write the composition from right to left so that $\lambda\mu$ is defined whenever $s(\lambda)=r(\mu)$. 

We say that $\Lambda$ is \emph{row-finite} if for every $v\in \Lambda^{(0)}$ and $n\in \N^k$, the set
\[
v\Lambda^n \coloneqq \{\lambda\in \Lambda \mid r(\lambda)=v,\ d(\lambda)=n\}
\]
is finite, and that $\Lambda$ has \emph{no sources} (or is \emph{source-free}) if $v\Lambda^n\not=\emptyset$ for all $v\in \Lambda^{(0)}$ and $n\in \N^k$. We shall assume throughout that all $k$-graphs are row-finite and source-free.

We denote by $\Omega_k$ the category associated to the opposite of the partially ordered set $\N^k$. Its objects are given by elements of $\N^k$ and its morphisms from $m$ to $n$ are given by 
\[
\Omega_k(n,m)=\begin{cases}
	\{(n,m)\},&n\leq m,\\
	\emptyset,&\text{else}.
\end{cases}
\]
The \emph{path space} $\Lambda^\infty$ of $\Lambda$ is the space of all functors $\lambda\colon \Omega_k\to \Lambda$, whose topology is generated by the basic compact open sets
\[
Z(\lambda)\coloneqq \{x\in \Lambda^\infty\mid x(0,d(\lambda))=\lambda\},\qquad \lambda\in \Lambda,
\]
see \cite[Lemma~2.6]{Kumjian2000}. For $p\in \N^k$, we denote by $\sigma^p\colon \Lambda^\infty\to \Lambda^\infty$ the shift map given by $\sigma^p(x)(n,m)=x(n+p,m+p)$. 

The \emph{path groupoid} $G_\Lambda\coloneqq G_\sigma$ associated to $\Lambda$ is the Deaconu--Renault groupoid associated to the action $\sigma\colon \N^k\curvearrowright \Lambda^\infty$. 

Following \cite{Pask2004}, the \emph{fundamental groupoid} $\Pi(\Lambda)$ of $\Lambda$ is the unique (up to isomorphism) small groupoid together with a functor $i_\Lambda\colon \Lambda \to \Pi(\Lambda)$ such that $i_\Lambda$ is bijective on objects and such that every functor from $\Lambda$ into a discrete groupoid uniquely factors through $i_\Lambda$. Concretely, one may define $\Pi(\Lambda)$ as the topological fundamental groupoid $\Pi_1(B\Lambda)$ of the classifying space of $\Lambda$. The functor $i_\Lambda$ need not be injective in general \cite{Pask2004}, but when it \emph{is}, we say that \emph{$\Lambda$ embeds into its fundamental groupoid}. 

The following result is inspired by Xin Li's presentation of graph groupoids in terms of partial actions by free groups \cite{Li2017}.

\begin{prop}\label{prop-k-graph}
Let $\Lambda$ be a row-finite, source-free higher rank graph which embeds into its fundamental groupoid. Then there exists a locally proper pure cocycle $c\colon G_\Lambda \to \Gamma$ into a discrete group. In particular, $G_\Lambda$ is a transformation groupoid by a partial action with closed graphs. 
\end{prop}

For the proof, we need the following lemma which is similar to \cite[Proposition~3.12]{Brownlowe2025a}.

\begin{lemma}\label{lem-groupoidfunctor}
Let $G$ be a discrete groupoid and $g,h\in G$ such that $s(g)=s(h)$ and $g\not=h$. Then there exists a homomorphism $c\colon G\to \Gamma$ into a discrete group satisfying $c(g)\not=c(h)$. 
\end{lemma}

\begin{proof}
By \cite[Corollary~3.10]{deCastroKang2024}, there is a pure cocycle $c\colon G\to \Gamma$ into a discrete group. 
Assume that $c(g)=c(h)$. Then we have
\[
c(gh^{-1})=c(g)c(h)^{-1}=1.
\]
By purity, we must have $gh^{-1}\in G^{(0)}$, so that $g=h$. This is a contradiction.
\end{proof}

\begin{proof}[Proof of \autoref{prop-k-graph}]
Following \cite[page~577]{Bridson1999}, we denote by $F(\Lambda)$ the group generated by symbols $[\lambda]$ for $\lambda\in\Lambda$ subject to the relations
\[
[\lambda][\mu]=[\lambda\mu]
\]
whenever $\lambda,\mu$ are composable. We define a cocycle
\[
c\colon G_\Lambda\to F(\Lambda),\qquad 
c(x,n,y)\coloneqq [x(0,m)][y(0,l)]^{-1},
\]
where $m,l\in \N^k$ are such that $m-l=n$ and $\sigma^m(x)=\sigma^l(y)$.

To see that $c$ is well-defined, let $(m,l)$ and $(m',l')$ be two such choices. Choose $r\in \N^k$ with $r\geq m,m'$. Then also $r-n\geq l,l'$, and both pairs can be enlarged to $(r,r-n)$, so it suffices to check invariance under replacing $(m,l)$ by $(m+p,l+p)$. Using the relations in $F(\Lambda)$ and $\sigma^m(x)=\sigma^l(y)$, we obtain
\begin{align*}
[x(0,m+p)][y(0,l+p)]^{-1}&=[x(0,m)][x(m,m+p)][y(l,l+p)]^{-1}[y(0,l)]^{-1}\\
&=[x(0,m)][y(0,l)]^{-1}.
\end{align*}
Thus $c$ is well-defined.

To see that $c$ is continuous, note that it is constant on sets of the form
\[
Z(\lambda,\mu)=\{(x,d(\lambda)-d(\mu),y)\mid x(0,d(\lambda))=\lambda,\ y(0,d(\mu))=\mu,\ \sigma^{d(\lambda)}(x)=\sigma^{d(\mu)}(y)\},
\]
which form a basis of compact open sets by \cite[Proposition~2.8]{Kumjian2000}.

To see that $c$ is a groupoid homomorphism, let $(x,n,y)(y,m,z)\in G_\Lambda$ and choose $l,p,q\in \N^k$ such that $n=l-p$, $m=p-q$, and $\sigma^l(x)=\sigma^p(y)=\sigma^q(z)$. Then
\begin{align*}
c(x,n,y)c(y,m,z)
&=[x(0,l)][y(0,p)]^{-1}[y(0,p)][z(0,q)]^{-1}\\
&=[x(0,l)][z(0,q)]^{-1}
=c(x,n+m,z).
\end{align*}

To see that $c$ is pure, suppose $c(x,n,y)=1$. Since the degree functor $d$ induces a homomorphism $\tilde d\colon F(\Lambda)\to \Z^k$ with $\tilde d([\lambda])=d(\lambda)$, we obtain $n=0$. Choose $m\in \N^k$ with $\sigma^m(x)=\sigma^m(y)$. In particular $x(0,m)$ and $y(0,m)$ have the same source. 
By assumption, we have
\[
[x(0,m)][y(0,m)]^{-1}=c(x,n,y)=1,
\]
so $[x(0,m)]=[y(0,m)]$ in $F(\Lambda)$.

Assume that the images of $x(0,m)$ and $y(0,m)$ in the fundamental groupoid $\Pi(\Lambda)$ are different. Then by \autoref{lem-groupoidfunctor}, there exists a homomorphism $\varphi\colon \Pi(\Lambda)\to \Gamma$ into a discrete group such that
\[
\varphi(i_\Lambda(x(0,m)))\not=\varphi(i_\Lambda(y(0,m))).
\]
Since $\varphi\circ i_\Lambda\colon \Lambda\to \Gamma$ is a functor, it induces a group homomorphism $\widetilde{\varphi}\colon F(\Lambda)\to \Gamma$, contradicting $[x(0,m)]=[y(0,m)]$. Hence the images of $x(0,m)$ and $y(0,m)$ in $\Pi(\Lambda)$ agree, and since $\Lambda$ embeds into $\Pi(\Lambda)$, we conclude $x(0,m)=y(0,m)$. By \cite[Proposition~2.3]{Kumjian2000}, this implies $x=y$, and therefore $c$ is pure.

Finally, to see that $c$ is locally proper, let $f=r\times c\times s$. For $\lambda,\mu\in \Lambda$, one checks that
\[
f^{-1}\big(Z(\lambda)\times \{[\lambda][\mu]^{-1}\}\times Z(\mu)\big)=Z(\lambda,\mu),
\]
where $Z(\lambda,\mu)$ is the compact set defined after \cite[Definition~2.7]{Kumjian2000}. Since every point in the codomain admits such a neighbourhood, a standard compactness argument shows that $f$ is proper.

The conclusion now follows from \autoref{thm:deCastroKang}.
\end{proof}

The following theorem shows that the total disconnectedness assumption in \autoref{rem:k-graph} cannot be dropped. 
In the following, we identify a single local homeomorphism $\sigma\colon X\to X$ with an $\N$-action $\sigma\colon \N\curvearrowright X$ via iteration.
\begin{theorem}\label{prop:no_pure_cocycle_connected}
Let $\sigma\colon X\to X$ be a local homeomorphism of a locally compact Hausdorff space. 
Suppose there exists a connected subset $C\subset X$ such that $\sigma|_C$ is not injective.
Then the Deaconu--Renault groupoid $G_\sigma$ admits no pure cocycle into any discrete group. 
In particular, $G_\sigma$ is not isomorphic to a transformation groupoid
arising from a partial action of a discrete group.
\end{theorem}
\begin{proof}
Assume by contradiction that there is a pure cocycle
$
c\colon G_\sigma\to\Gamma
$
into a discrete group $\Gamma$.
Let
\[
A_C=\{(x,1,\sigma(x))\mid x\in C\},
\]
which is a connected set being the image of $C$ under the continuous map 
\[\varphi\colon C\to A_C,\quad  x\mapsto (x,1,\sigma(x)).\]
Since $\Gamma$ is discrete and $c$ is continuous,
the image of a connected set is a single point, hence
$
A_C\subseteq c^{-1}(\gamma),
$
for some $\gamma \in \Gamma$. 

By \autoref{thm:deCastroKang}, $G_\gamma\coloneqq c^{-1}(\gamma)$ is a bisection. In particular,
\[
s|_{c^{-1}(\gamma)}\colon G_\gamma\to G^{(0)}
\]
is injective.
However,
\[
s|_{c^{-1}(\gamma)}\circ \varphi(x)=s(x,1,\sigma(x))=\sigma(x),
\]
for all $x\in C$. Since $\varphi$ is injective as well, we conclude that $\sigma|_C$ is injective, contradicting our assumptions.
\end{proof}

\begin{exa}\label{exa-doublingmap}
	Denote by $S^1\subset \C$ the unit circle.
	It follows from \autoref{prop:no_pure_cocycle_connected} that the Deaconu--Renault groupoid $G_\sigma$ associated to $\sigma\colon S^1\to S^1$,  $z\mapsto z^2$ is not a transformation groupoid by a partial group action. 
	In particular, there exists a Hausdorff \'etale groupoid with property $\bar\Delta$ which is not a transformation groupoid by a partial group action. It follows from \cite{Exel2011} that in this particular example, the reduced $C^*$-algebra $C^*_r(G_\sigma)$ is even a UCT Kirchberg algebra. 
\end{exa}

\begin{rem}
The Deaconu--Renault examples discussed above are of a rather different
nature from the ample examples considered in \autoref{sec:HLS}.
Indeed, they satisfy property $\bar\Delta$, and hence are inner amenable
in the sense of \autoref{sec:inner-amen}, and in many cases their $C^*$-algebras
are classifiable Kirchberg algebras. Nevertheless, they may fail to be
transformation groupoids of partial actions. This shows that the
obstruction to being a transformation groupoid is independent both from
inner amenability and from regularity/classifiability phenomena on the
$C^*$-algebraic side.
\end{rem}

\section{Coarse groupoids}\label{sec:coarse-groupoids}
In this section, we investigate when the \emph{coarse groupoid} associated to a uniformly locally finite metric space (or more generally connected coarse space) is isomorphic to a transformation groupoid associated to a partial action of a discrete group. We prove that the coarse groupoid can be realized as a transformation groupoid by a partial group action if and only if the space admits a coarse injection into the given group and that the coarse injection can be chosen to be a coarse embedding if and only if the partial action can be chosen to have clopen domains.

Coarse groupoids associated to uniformly locally finite metric spaces were originally introduced as groupoid models for their uniform Roe algebras \cite{Roe2003}. 
Recall that a metric space $X$ is called \emph{uniformly locally finite} if its
balls of a fixed radius have uniformly bounded cardinality, that is
\[
\sup_{x\in X}|B_R(x)|<\infty,
\qquad \text{for all } R>0.
\]
A basic example of a uniformly locally finite metric space is given by a finitely
generated group $\Gamma$ with a finite generating set
$S=S^{-1}\cup\{1\}$. The distance between two points $\gamma,\eta\in\Gamma$
is given by the word metric
\[
d_S(\gamma,\eta)
=
\min\{l\ge0 \mid \gamma\eta^{-1}=s_1\dots s_l,\ s_1,\dots,s_l\in S\},
\]
where we define the empty product as $1$ so that $d_S(\gamma,\gamma)=0$.

One of the many insights of Roe \cite{Roe2003} is that large-scale properties
of a metric space do not depend on the precise values of the metric, but only
on the collection of neighbourhoods of the diagonal in $X\times X$.
These sets are often called \emph{entourages}. This observation leads to the
axiomatic notion of a coarse structure which allows one to define Roe algebras and coarse groupoids in a more general setup.

Our main motivation for using this additional level of abstraction is that it
allows us to study large-scale properties of groups which are not finitely
generated.
\begin{defi}[{\cite[Definition~2.3]{Roe2003}}]
A \emph{coarse space} is a pair $(X,\mathcal C_X)$ where $X$ is a set and
$\mathcal C_X$ is a collection of subsets of $X\times X$ such that

\begin{enumerate}

\item $\Delta_X\in \mathcal C_X$;

\item if $E\in \mathcal C_X$, then
\[
E^{-1}=\{(x,y)\mid (y,x)\in E\}\in \mathcal C_X;
\]

\item if $E_1,E_2\in \mathcal C_X$, then
\[
E_1\circ E_2
=
\{(x,z)\mid (x,y)\in E_1,\ (y,z)\in E_2\}
\in \mathcal C_X;
\]

\item if $E_1,E_2\in \mathcal C_X$, then $E_1\cup E_2\in\mathcal C_X$;

\item if $E_2\subset E_1$ and $E_1\in\mathcal C_X$, then $E_2\in\mathcal C_X$.

\end{enumerate}

A coarse space $(X,\mathcal C_X)$ is called \emph{uniformly locally finite}
if for every $E\in\mathcal C_X$ we have
\[
\sup_{x\in X}|E[x]|<\infty,
\qquad
E[x]=\{y\in X\mid (x,y)\in E\}.
\]

A coarse space $(X,\mathcal C_X)$ is called \emph{connected} if for every
$x,y\in X$ there exists $E\in\mathcal C_X$ such that $(x,y)\in E$.
\end{defi}

We call $\mathcal C_X$ the \emph{coarse structure} of $X$ and simply say
that $X$ is a coarse space when the coarse structure is understood.
The elements of $\mathcal C_X$ are called \emph{entourages}.
For a given set $X$ and a subset $A\subset\mathcal P(X\times X)$ there is
always a smallest coarse structure containing $A$. We call this the
\emph{coarse structure generated by $A$}.

\begin{exa}
The following pairs $(X,\mathcal C_X)$ are uniformly locally finite
connected coarse spaces.

\begin{enumerate}

\item
A uniformly locally finite metric space $X$ with the coarse structure
generated by the sets
\[
E_R=\{(x,y)\in X\times X\mid d(x,y)<R\},
\qquad R>0.
\]

\item
A discrete group $\Gamma$ with the coarse structure generated by the sets
\[
E_S
=
\{(\gamma,\eta)\in\Gamma\times\Gamma\mid \gamma\eta^{-1}\in S\},
\]
for all finite subsets $S\subset\Gamma$.

\end{enumerate}

Unless stated otherwise, we equip all uniformly locally finite metric
spaces and all discrete groups with the coarse structures described above.

\end{exa}
We denote the Stone-{\v{C}}ech compactification of a discrete set $X$ by $\beta X$. 
\begin{defi}[{\cite[Definition 10.19]{Roe2003}}]
Let $X$ be a uniformly locally finite coarse space.
The \emph{coarse groupoid} associated to $X$ is defined by
\[
G(X)
=
\bigcup_{E\in\mathcal C_X}\overline{E}
\subset
\beta (X\times X),
\]
and equipped with the subspace topology. 
Its unit space is
\[
G(X)^{(0)}=\beta X\cong \beta(\Delta_X)\subset \beta(X\times X),
\]
and the range, source and multiplication maps are the unique continuous extensions of the maps
\[
r(x,y)=x,
\qquad
s(x,y)=y,
\qquad
(x,y)(y,z)=(x,z),
\]
for all $x,y,z\in X$ with $(x,y),(y,z)\in G(X)$ (see \cite[Theorem 10.20]{Roe2003}).
\end{defi}

\begin{defi}
A map $f\colon X\to Y$ between coarse spaces is called \emph{coarse}
if it satisfies $(f\times f)(\mathcal C_X)\subset \mathcal C_Y$, i.e. images of entourages are entourages. 
A coarse map $f$ is called a \emph{coarse embedding}, if we additionally have $(f\times f)^{-1}(\mathcal C_Y)\subset \mathcal C_X$, i.e. preimages of entourages are entourages. 
A coarse map which is injective as a map is called a \emph{coarse injection}.
\end{defi}

\begin{theorem}\label{thm:coarse}
Let $X$ be a uniformly locally finite connected coarse space and let
$\Gamma$ be a discrete group. Then the following are equivalent:

\begin{enumerate}
\item\label{item-coarse1} $X$ admits a coarse injection $X\hookrightarrow\Gamma$;

\item\label{item-coarse2} $G(X)$ is isomorphic to a transformation groupoid
\[
\Gamma\ltimes\beta X
\]
arising from a partial action $\Gamma\curvearrowright\beta X$.
\end{enumerate}
If the above conditions are satisfied, then the following are equivalent as well:
\begin{enumerate}
    \stepcounter{enumi}
    \stepcounter{enumi}
    \item\label{item-coarse1c} The coarse injection $X\hookrightarrow \Gamma$ can be chosen to be a coarse embedding;
    \item\label{item-coarse2c} The partial action $\Gamma\curvearrowright\beta X$ can be chosen to have clopen domains. 
\end{enumerate}
\end{theorem}

\begin{rem}(cf.~\cite[Proof of Theorem~4]{Brodzki2007})
A uniformly locally finite coarse space $X$ admits a coarse embedding
$f\colon X\to\Gamma$ into a group if and only if it admits an
\emph{injective} coarse embedding into a group coarsely equivalent
to $\Gamma$.

Indeed, since $f$ is a coarse embedding the set
$f^{-1}(\Delta_\Gamma)$ is controlled, hence
\[
n=\sup_{\gamma\in\Gamma}|f^{-1}(\gamma)|<\infty
\]
since $X$ is uniformly locally finite.
We may therefore choose a map
\[
\ell\colon X\to\mathbb Z/n\mathbb Z
\]
which is injective on $f^{-1}(\gamma)$ for each $\gamma \in \Gamma$.
It follows that
\[
f\times\ell\colon X\to\Gamma\times\mathbb Z/n\mathbb Z
\]
is an injective coarse embedding.
\end{rem}

\begin{rem}
If $X$ admits a coarse embedding into a group $\Gamma$, then $X$ is
automatically connected as a coarse space, since groups are connected
with respect to their canonical coarse structures.
We do not know if every uniformly locally finite metric coarsely embeds or even coarsely injects into a discrete group.
\end{rem}

\begin{proof}[Proof of \autoref{thm:coarse}]
We first prove \ref{item-coarse2} $\Rightarrow$ \ref{item-coarse1}. 
Assume that 
\[
G(X)\cong \Gamma\ltimes\beta X
\]
for a partial action.
Denote by
\begin{equation}\label{eq-c}
c\colon G(X)\to\Gamma
\end{equation}
the associated pure cocycle as in \autoref{thm:deCastroKang}. 
Fix $x_0\in X$ and define
\begin{equation}\label{def-f}
f\colon X\to\Gamma,
\qquad
f(x)=c(x,x_0).
\end{equation}
Since $X$ is connected, we have $(x,x_0)\in G(X)$ for all $x\in X$,
so that $f$ is well defined.
If $f(x)=f(y)$ for some $x,y\in X$, then
\[
c(x,y)=c(x,x_0)c(x_0,y)=f(x)f(y)^{-1}=1,
\]
and since $c$ is pure this implies $x=y$.
Hence $f$ is injective.

We now show that $f$ is coarse.
Let $A\in \mathcal C_X$ be an entourage.
Then $\overline A$ is compact in $G(X)$ and therefore
$c(\overline A)$ is finite. Let
\[
S=c(\overline A)\subset\Gamma .
\]
For $(x,y)\in A$ we have
\(
f(x)f(y)^{-1}=c(x,y)\in S,
\)
so
\[
(f\times f)(A)\subset
E_S=
\{(\gamma,\eta)\in\Gamma\times\Gamma\mid \gamma\eta^{-1}\in S\},
\]
which is an entourage in $\Gamma$.

\medskip

We now prove \ref{item-coarse1} $\Rightarrow$ \ref{item-coarse2}. 
Assume that there exists a coarse injection
\[
f\colon X\hookrightarrow\Gamma.
\]
The left translation action $L\colon\Gamma\curvearrowright \Gamma$ of $\Gamma$ on itself given by 
\[L_\gamma\colon \Gamma \to \Gamma,\quad L_\gamma(\eta)=\gamma\eta,\quad \gamma\in \Gamma\]
extends to a
continuous action $L\colon \Gamma\curvearrowright \beta\Gamma$ on the Stone--Čech compactification of $\Gamma$.
Since $Y\coloneqq f(X)\subset \Gamma$ is discrete, the space $\beta Y$ can be identified
with the closure of $Y$ in $\beta\Gamma$.
Restricting the global action of $\Gamma$ on $\beta\Gamma$ to $\beta Y$
defines a partial action of $\Gamma$ on $\beta Y$ with clopen domains
\[
D_{\gamma^{-1}}
=
\beta Y\cap L_\gamma^{-1}(\beta Y), \qquad \gamma\in \Gamma.
\]
Note that there is a canonical isomorphism
\[\Gamma\ltimes \beta\Gamma \xrightarrow{\cong} G(\Gamma),\quad (\gamma,\xi)\mapsto (L_\gamma (\xi),\xi)\]
of topological groupoids (compare \cite[Example 10.26]{Roe2003}). 
By reducing this isomorphism to $\beta Y\subset \beta \Gamma$ (see \eqref{eq-reduction}), we obtain an isomorphism 
\begin{equation}\label{eq:iso-coarse}
G(\Gamma)(\beta Y)\xrightarrow{\cong} \Gamma\ltimes \beta Y.
\end{equation}
Since $f$ is coarse, it follows that the continuous extension
\(
\beta(f\times f)\colon \beta(X\times X)\to \beta(Y\times Y)
\)
satisfies
\begin{equation}\label{eq-GXinGamma}
    \beta(f\times f)(G(X))\subset G(\Gamma)(\beta Y)
\end{equation} 
Indeed, for every entourage $E\in \mathcal C_X$, the set $(f\times f)(E)$ is an
entourage in $\Gamma$, so
\[
\beta(f\times f)(\overline E)\subset \overline{(f\times f)(E)}\subset G(\Gamma).
\]
Taking the union over all entourages $E\in\mathcal C_X$, we obtain
\(
\beta(f\times f)(G(X))\subset G(\Gamma).
\)
Since $f(X)=Y$, the image is in fact contained in the reduction $G(\Gamma)(\beta Y)$.

Composing the map
\(
\beta(f\times f)\colon G(X)\to G(\Gamma)(\beta Y)
\)
with the inverse of the isomorphism~\eqref{eq:iso-coarse} we obtain a continuous groupoid homomorphism (not necessarily an isomorphism)
\begin{equation}\label{eq-Phi}
\Phi\colon G(X)\to \Gamma\ltimes \beta Y.
\end{equation}
Since $f\colon X\to Y$ is a bijection, the induced map $\beta f\colon \beta X\to \beta Y$
is a homeomorphism, so units correspond to units under the reduction
$G(\Gamma)(\beta Y)$. Hence $\Phi$ is a \emph{pure} homomorphism of groupoids in the sense that $\Phi(g)$ is a unit if and only if $g\in G(X)$ is a unit. 
Denote by $\tilde c\colon \Gamma\ltimes \beta Y\to \Gamma$ the pure cocycle from \autoref{thm:deCastroKang}.
It follows that $\tilde c\circ \Phi\colon G(X)\to \Gamma$ is a pure cocycle too. In particular, $G(X)$ is a transformation groupoid by a partial action $\Gamma\curvearrowright \beta X$.\footnote{Note that this partial action may be different from the partial action obtained by restricting the right translation action of $\Gamma\curvearrowright \beta \Gamma$ to $\beta Y$.}

\medskip 

We now prove \ref{item-coarse2c} $\Rightarrow$ \ref{item-coarse1c}.
Assume that $G(X)\cong \Gamma\ltimes \beta X$ for a partial action with clopen domains. 
We show that the map $f$ constructed in \eqref{def-f} is not only a coarse injection but in fact a coarse embedding.
To see this, let $B\in\mathcal C_\Gamma$ be an entourage.
Then $B\subset E_S$ for some finite set $S\subset\Gamma$.
Note that the cocycle $c$ defined in \eqref{eq-c} is proper by \autoref{thm:deCastroKang} since $\Gamma\curvearrowright \beta X$ has clopen domains. 
Thus, for each $\gamma\in S$, the set $c^{-1}(\gamma)\subset G(X)$ is compact.
It follows that  
\[c^{-1}(\gamma)\subset \overline{E_1}\cup\dotsb \cup\overline{E_n}=\overline{E_1\cup\dotsb\cup E_n},\] 
for finitely many entourages $E_1,\dotsc,E_n\in \mathcal C_X$. 
In particular, the set
\[
E_\gamma\coloneqq c^{-1}(\gamma)\cap (X\times X)\subset E_1\cup\dotsb\cup E_n
\]
is an entourage in $X$.
Hence
\[
(f\times f)^{-1}(B)\subset
\bigcup_{\gamma\in S}E_\gamma
\]
is an entourage as well, proving that $f$ is a coarse embedding.

Next, we prove \ref{item-coarse1c} $\Rightarrow$ \ref{item-coarse2c}. 
Assume that the map $f\colon X\to \Gamma$ is not only a coarse injection, but in fact a coarse embedding. 
In this case the %
map $\Phi$ constructed in \eqref{eq-Phi} is an isomorphism. 
More precisely, let $Y\coloneqq f(X)\subseteq \Gamma$. Since $f$ is injective and a coarse embedding, it induces a bijection $f:X\to Y$ which is a coarse equivalence between $X$ and $Y$ (endowed with the coarse structure induced from $\Gamma$), hence $G(X)\cong G(Y)$. 
Now $G(Y)$ identifies with the reduction $G(\Gamma)(\beta Y)$ of the coarse groupoid $G(\Gamma)\cong \Gamma\ltimes\beta\Gamma$ to $\beta Y$, which is precisely the transformation groupoid $\Gamma\ltimes \beta Y$ associated to the partial action obtained by restricting the $\Gamma$-action on $\beta\Gamma$ to the compact open subspace $\beta Y$. 
In particular, $G(X)$ is isomorphic (via $\Phi$) to the transformation groupoid associated to a partial action $\Gamma\curvearrowright \beta X$ with clopen domains.
\end{proof}

The following example was inspired by discussions with Tim de Laat, who pointed us towards the maximal uniformly locally finite coarse structure.%
\begin{exa}\label{exa-maximalcoarse}
    Let $X$ be an infinite set. Equip $X$ with the maximal uniformly locally finite coarse structure $\mathcal C^{\max}_X$ defined by declaring a set $E\subset X\times X$ to be an entourage if and only if 
    \[\sup_{x\in X}|\{y\in X\colon (x,y)\in E \text{ or } (y,x)\in E\}|<\infty.\]
    Then $X$ does not coarsely inject into a discrete group.
    In particular, the coarse groupoid $G(X)$ is not a transformation groupoid by a partial action of a discrete group $\Gamma$. 
\end{exa}
\begin{proof}
    Let $f\colon X\hookrightarrow \Gamma$ be an injective map into a discrete group. 
    We prove that $f$ is not coarse. 
    Let $x_1,x_2,x_3,\dotsc\in X$ be pairwise distinct elements. 
    We recursively define an increasing sequence of integers $0<n_1<n_2<\dotsb$ as follows. 
    Let $n_1=1$. Assume that $n_1,\dotsc,n_k$ are already chosen for some $k$. 
    Then we pick $n_{k+1}>n_k$ in such a way that 
    \[f(x_{k+1})f(x_{n_{k+1}})^{-1}\notin \{f(x_1)f(x_{n_1})^{-1},\dotsc,f(x_k)f(x_{n_k})^{-1}\},\]
    which is possible since $X$ is infinite and $f$ is injective. 
    By definition, the set 
    \[E\coloneqq \{(x_k,x_{n_k})\mid k\in \N\}\subset X\times X\]
    is an entourage for the maximal uniformly locally finite coarse structure $\mathcal C^{\max}_X$. 
    However, since the set 
    \[S\coloneqq \{f(x)f(y)^{-1}\mid (x,y)\in E\}\subset \Gamma\]
    is infinite, we conclude that $f(E)\subset \Gamma\times \Gamma$ is not an entourage. 
\end{proof}

The following example shows that a coarse groupoid may admit both
proper and non-proper pure cocycles. Equivalently, the same groupoid may
arise from partial actions with clopen domains and from partial actions
whose domains are open but not closed.

Thus the property that the domains are clopen is not intrinsic to the
groupoid itself but depends on the chosen realization as a
transformation groupoid.
\begin{exa}\label{exa:coarse_counterexample}
Let $X=\N$ and consider the coarse structure
\[
\mathcal C_X
=
\{E\subset \N\times \N \mid E\setminus \Delta_\N \text{ is finite}\}.
\]
Then $(X,\mathcal C_X)$ is uniformly locally finite and connected. %
We construct two different realizations of $G(X)$ as a transformation groupoid indicating that the properness of the cocycle (equivalently the closedness of the domains of the partial action) is not an invariant of the underlying groupoid. 
\begin{enumerate}
    \item 
A non-proper pure cocycle is given by 
\[c\colon G(X)\to \Z,\quad c(n,m)\coloneqq\begin{cases}0,&n=m\in \beta\N\\ n-m,&n,m\in \N.\end{cases}\]
Note that this cocycle is not proper since $c^{-1}(1)= \{(n,n+1)\mid n\in \N\}$ is not compact. 
The partial action $\theta\colon \Z\curvearrowright \beta \N$
is given by 
\[D_{k}=\{k,k+1,k+2,\dotsc\},\qquad D_{-k}=\N,\qquad \theta_k(n)=n+k,\]
for $k\in \N\setminus \{0\}$ and $n\in D_{-k}$.

\item 
A proper pure cocycle of $G(X)$ with values in the free group $\Gamma=\langle a,b\rangle$ is given by 
\[c\colon G(X)\to \Gamma,\quad c(n,m)=\begin{cases}
    1,&n=m\in \beta\N\\
    b^{-m}a^{n-m}b^{n},& n,m\in \N.
\end{cases}\]
Since the term $b^{-m}a^{n-m}b^{n}$ determines $n,m\in \N$ uniquely, the preimages $c^{-1}(\gamma)$ for $\gamma\in \Gamma$ are either $\beta \N$ for $\gamma=1$ or contain at most one point for $\gamma\not=1$. In particular, $c$ is proper and pure. 
\end{enumerate}
\end{exa}

\section{Inner amenability}\label{sec:inner-amen}

In this section, we provide the first examples of Hausdorff \'etale groupoids which fail to be inner amenable in the sense of \cite{AnantharamanDelaroche}. 

\begin{defi}[Positive type function]\label{def:pos-type-groupoid}
A function $\varphi\colon G\to \C$ is of \emph{positive type} (or \emph{positive definite})
if for every $u\in G^{(0)}$, every $m\in \N$ and every $g_1,\dots,g_m\in G_u$,
the matrix
\[
\bigl(\varphi(g_i g_j^{-1})\bigr)_{i,j=1}^m
\]
is positive in $M_m(\C)$.
\end{defi}

\begin{defi}[{\cite[Definition~5.4]{AnantharamanDelaroche}}]\label{def:proper-support}
A function $\varphi\colon G\times G\to \C$ has \emph{proper support} if for every compact
$K\subset G^{(0)}$ and every compact $C\subset G(K)$, the sets
\[
\supp(\varphi)\cap (G(K)\times C)\quad\text{and}\quad \supp(\varphi)\cap (C\times G(K))
\]
are relatively compact in $G\times G$.
\end{defi}

\begin{defi}[{\cite[Definition 5.5]{AnantharamanDelaroche}}]\label{def:inner-amenable}
A locally compact Hausdorff groupoid $G$ is called \emph{inner amenable} if for every compact subset $C\subset G$ and every $\varepsilon>0$, there exists a properly supported positive type function $\varphi\colon G\times G\to \C$ such that $\sup_{x,y\in G^{(0)}}\varphi(x,y)\leq 1$ and $\sup_{g\in C}|\varphi(g,g)-1|<\varepsilon$. 
\end{defi}
We compare the definition of inner amenability with the definition of amenability:
\begin{defi}[{\cite{Renault1980,AnantharamanDelaroche2000}}]
	A locally compact Hausdorff groupoid $G$ is called (topologically) \emph{amenable} if for every compact subset $C\subset G$ and every $\varepsilon>0$, there exists a compactly supported positive type function $\varphi\colon G\to \C$ such that $\sup_{x\in G^{(0)}}\varphi(x)\leq 1$ and $\sup_{g\in C}|\varphi(g)-1|<\varepsilon$. 
\end{defi}

\begin{exa}$~$\label{exa:inneramenable}
\begin{enumerate}
\item\label{exaitem1}
Every amenable groupoid is inner amenable. This can be seen by associating
to a positive type function $\varphi\colon G\to \C$ the positive type function
\[
G\times G\to \C,\qquad (g,h)\mapsto \varphi(g)\varphi(h)
\]
(c.f.\ \cite[paragraph after Definition~5.5]{AnantharamanDelaroche}).

\item\label{exaitem2}
Let $\Gamma \curvearrowright X$ be a partial action of a discrete group $\Gamma$
on a locally compact space $X$ which admits a Hausdorff globalization
(see \cite[\S 5]{Exel2017}). Equivalently, assume that the graphs $\{(x,\gamma x)\mid x\in D_{\gamma^{-1}}\}\subset X\times X$ are closed for all $\gamma\in \Gamma$. 
Then $\Gamma \ltimes X$ is inner amenable by
\cite[Corollary~5.18]{AnantharamanDelaroche}. This, in particular, applies
to all partial actions with clopen domains.
\end{enumerate}
\end{exa}

\begin{rem}
We do not know whether all partial transformation groupoids
(with possibly non-closed graphs) are inner amenable.
\end{rem}
As a consequence of Theorem~\ref{thm:coarse}, we immediately obtain:
\begin{cor}
If $X$ is a uniformly locally finite connected coarse space that coarsely
embeds into a group, then its coarse groupoid $G(X)$ is inner amenable.
\end{cor}

We now state the main result of this section which gives the first examples of non inner amenable groupoids. 
The strategy of proof is to approximate the fiber at infinity by compact fibers in order to upgrade a properly supported positive type function to a compactly supported positive type function, thus upgrading inner amenability to amenability. 
Throughout, we use the notation of \autoref{sec:HLS}.
\begin{theorem}\label{thm:notinneramenable}
	Let $(\Gamma,(N_k)_{k\in \N})$ be an approximated group. 
	Denote by $G=G_{\mathrm{HLS}}(\Gamma,(N_k)_{k\in \N})$ the associated HLS groupoid and by $G\ltimes \widehat X=G_{\mathrm{AFS}}(\Gamma,(N_k)_{k\in \N})$ the associated AFS groupoid. 
	Then the following are equivalent:
	\begin{enumerate}
		\item\label{thmitem1} $\Gamma$ is amenable;
		\item\label{thmitem2} $G$ is inner amenable;
		\item\label{thmitem3} $G\ltimes \widehat X$ is inner amenable.
	\end{enumerate} 
	In particular, there are locally compact Hausdorff \'etale (and ample and principal or non-principal) groupoids which are not inner amenable. 
\end{theorem}

\begin{proof}
	The implication $\ref{thmitem1}\Rightarrow\ref{thmitem2}$ follows from \cite[Lemma 2.3]{Willett2015} and \autoref{exa:inneramenable}\ref{exaitem1}. The implication $\ref{thmitem2}\Rightarrow\ref{thmitem3}$ follows from \cite[Proposition 5.6]{AnantharamanDelaroche} and the fact that the projection map $\rho\colon G\ltimes \widehat X\to G$ is locally proper. It remains to prove $\ref{thmitem3}\Rightarrow \ref{thmitem1}$. 

	Suppose that $G\ltimes \widehat X$ is inner amenable. 
	We first prove that $\Gamma\ltimes \widehat X_\infty$ is amenable. 
	Let $\varepsilon>0$ and let $C\subset \Gamma\ltimes \widehat X_\infty$ be a compact subset.
	We may assume that there is a finite subset $F\subset \Gamma$ such that 
	\[C\coloneqq \{(\infty,\gamma,x)\mid \gamma\in F, x\in \widehat X_\infty\}.\]
	Since $G\ltimes X$ is inner amenable, we can find a properly supported positive type function $\varphi\colon (G\ltimes \widehat X)\times (G\ltimes \widehat X)\to \C$ such that 
	\begin{equation}\label{varphiproperties}
		 \sup_{x,y\in \widehat X}\varphi(x,y)\leq 1 \qquad \text{ and }\qquad
	 \sup_{g\in C}|\varphi(g,g)- 1|< \varepsilon,
	\end{equation}
	where $C$ is considered as a subset of $G\ltimes\widehat X$. 
	Since $\varphi$ is continuous, there is an open neighbourhood $\Delta_C\coloneqq \{(g,g)\mid g\in C\}\subset U\subset (G\ltimes \widehat X)^{\times 2}$ such that 
	\begin{equation}\label{eq:smallonU}
		\sup_{(g,h)\in U}|\varphi(g,h)-1|<\varepsilon.
	\end{equation}
	Denote by $\widehat \pi_k\colon \Gamma\ltimes \widehat X_\infty\to \Gamma_k\ltimes\widehat X_k$ the quotient map and note that the set
	\[A\coloneqq \bigcup_{k\in \N\cup\{\infty\}}(\widehat \pi_k\times \id)(C)\setminus U\subset (G\ltimes \widehat X)^{\times 2}\] 
	is compact being the complement of an open set inside a continuous image of $(\N\cup \{\infty\})\times C$.
	Thus, denoting by $\mathrm{pr}_1\colon (G\ltimes \widehat X)^{\times 2}\to G\ltimes \widehat X$ the projection to the first factor, the set $(r\circ \rho\circ \mathrm{pr}_1)(A)\subset \N\cup \{\infty\}$ is a compact subset which does not contain $\infty$ since $\Delta_C\subset U$. 
	In particular, there exists an integer $n\in \N\setminus (r\circ \rho\circ \mathrm{pr}_1)(A)$. 
	It follows that 
	\begin{equation}\label{eq:UC}
		 (\widehat\pi_n\times \id)(C)\subset \bigcup_{k\in \N\cup\{\infty\}}(\widehat \pi_k\times \id)(C)\setminus A \subset U.
	\end{equation}
	The function
	\[\psi\colon \Gamma\ltimes \widehat X_\infty\to \C,\quad \psi(\gamma,x)=\varphi(\widehat \pi_n(\gamma,x),\gamma,x)\]
	is positive type being the composition of a positive type function with the groupoid homomorphism $\widehat\pi_n\times \mathrm{id}\colon \Gamma\ltimes \widehat X_\infty\to (G\ltimes \widehat X)^{\times 2}$. 
	Since $\varphi$ is properly supported and $\Gamma_n\ltimes \widehat X_n$ is compact, $\psi$ is compactly supported. 
	We have $\sup_{x\in \widehat X_\infty}\psi(x)\leq 1$ by \eqref{varphiproperties}, and $\sup_{g\in C}|\psi(g)-1|<\varepsilon$ by \eqref{eq:smallonU} and \eqref{eq:UC}. 
	Thus, $\Gamma\ltimes \widehat X_\infty$ is amenable. 
	Since the Haar measure on $\widehat X_\infty=\{\infty\}\times \widehat \Gamma$ is a $\Gamma$-invariant Borel probability measure, it follows that $\Gamma$ is amenable (c.f. \cite[Lemma 2.2]{Gardella2023}).  
\end{proof}

\begin{rem}
	The main motivation of Anantharaman-Delaroche's notion of inner amenability is \cite[Theorem 10.8]{AnantharamanDelaroche} which proves the equivalence of several different exactness conditions for inner amenable \'etale groupoids $G$, ranging from the formally strongest condition of \emph{strong amenability at $\infty$} to the formally weakest condition of \emph{exactness} of the reduced $C^*$-algebra $C^*_r(G)$.
	Although our \autoref{thm:notinneramenable} provides examples of groupoids which are not inner amenable, they still satisfy the theorem by Anantharaman-Delaroche.

	Indeed, following an argument of Kirchberg, Willett showed in \cite[Lemma 2.3]{Willett2015} that the HLS groupoid $G$ associated to an approximated group $(\Gamma,(N_k)_{k\in \N})$ has an exact reduced $C^*$-algebra $C^*_r(G)$ if and only if $\Gamma$ is amenable.
	The associated Alekseev--Finn-Sell groupoid $G\ltimes \widehat X$ behaves analogously: Since $C^*_r(G\ltimes \widehat X)$ contains $C^*_r(G)$ as a subalgebra, it is exact if and only if $C^*_r(G)$ is exact. 
\end{rem}

\begin{rem}
	We do not know if all coarse groupoids $G(X)$ associated to uniformly locally finite coarse spaces are inner amenable. However, coarse groupoids associated to uniformly locally finite metric spaces still satisfy Anantharaman--Delaroche's theorem \cite[Theorem 10.8]{AnantharamanDelaroche} since in this case, all the appearing exactness conditions can be characterized in \cite{Sako2020} in terms of Yu's \emph{property A} \cite{Yu2000}. 
\end{rem}

\section{Crossed product models for Kirchberg algebras}\label{sec:kirchberg}

In this section, we prove that every unital UCT Kirchberg algebra arises from a partial action of a discrete group on the Cantor set. 
This complements a theorem by Wu \cite{Wu2025} which represents every stable UCT Kirchberg algebra in terms of a global action of a discrete group $\Gamma$ on a totally disconnected locally compact Hausdorff space $X$. 
To prove the theorem, we show that in Wu's models, any class in $K_0(C_0(X)\rtimes_r \Gamma)$ can be represented by the characteristic function of a compact open subset of $X$. This allows us to choose our partial action models as reductions of Wu's actions to these subsets.
We achieve this by tracing through the computations of \cite{Brownlowe2025} to identify generators of $K_0(C_0(X)\rtimes_r \Gamma)$ with certain compact open subsets of $X$. Subsequently, we establish \emph{paradoxical comparison} in the sense of \cite{Ma2022} for $\Gamma\curvearrowright X$ in order to be able to represent sums and differences of the generators again by compact open subsets. 

Let $G$ be a Hausdorff \'etale groupoid. The associated \emph{full groupoid $C^*$-algebra} $C^*(G)$ is the enveloping $C^*$-algebra of the $*$-algebra $C_c(G)$ of compactly supported continuous functions on $G$ with respect to the product and involution given by 
\[f^*(g)=\overline{f(g^{-1})},\qquad f_1*f_2(g)=\sum_{h\in G^{r(g)}}f_1(h)f_2(h^{-1}g),\]
for $f,f_1,f_2\in C_c(G)$ and $g\in G$. 

There is an alternative completion of $C_c(G)$ called the \emph{reduced groupoid $C^*$-algebra} $C^*_r(G)$. 
But since all groupoids appearing in this section are amenable, we have canonical isomorphisms $C^*(G)=C^*_r(G)$, so that we are free to ignore the specific choice of completion. We refer to \cite{AnantharamanDelaroche2000} for more background on groupoid $C^*$-algebras and amenability. 

Note that the characteristic function $1_O$ of any compact open subset $O\subset G^{(0)}$ defines a projection in $C^*(G)$ and therefore a class $[1_O]$ in the $K$-theory group $K_0(C^*(G))$. 
The only properties of $K_0(C^*(G))$ which we will need are the following elementary facts:
\begin{lemma}\label{lem-elementary-K}
Let $G$ be a Hausdorff \'etale groupoid. 
\begin{enumerate}
	\item For any compact open bisection $U\subset G$, we have 
		\[[1_{s(U)}]=[1_{r(U)}]\in K_0(C^*(G)).\]
	\item For any disjoint compact open subsets $O_1,O_2\subset G^{(0)}$, we have 
		\[{[1_{O_1}]+[1_{O_2}]=[1_{O_1\cup O_2}]}\in K_0(C^*(G)).\]
\end{enumerate}
\end{lemma}
We will also need the following well-known elementary lemma.
\begin{lemma}\label{lem-full-corner}
	Let $A$ be a simple separable $C^*$-algebra and $p\in A$ a projection. Then the inclusion $pAp\hookrightarrow A$ is a $KK$-equivalence. In particular, it induces an isomorphism $K_*(pAp)\to K_*(A)$. 
\end{lemma}
\begin{proof}
	Apply \cite[Theorem~2.18]{Kasparov1988} to $\mathscr S=pA$. 
\end{proof}

\begin{defi}[{\cite[Definition 3.6]{Ma2022}}]
	A Hausdorff \'etale groupoid $G$ with totally disconnected unit space\footnote{Ma defines paradoxical comparison for general Hausdorff \'etale groupoids. Since the definition simplifies for totally disconnected unit spaces, we only give it in this special case.} is said to have \emph{paradoxical comparison} if for every nonempty compact open subset $O\subset G^{(0)}$, there are compact open bisections $U_1,U_2\subset G$ such that $s(U_1)=s(U_2)=O$ and such that $r(U_1),r(U_2)$ are disjoint subsets of $O$. 
\end{defi}

\begin{lemma}\label{lem-paradoxical}
	Let $G$ be a Hausdorff \'etale groupoid with totally disconnected unit space and paradoxical comparison. Then the set
	\[S\coloneqq \{[1_O]\mid O\subset G^{(0)} \text{compact open}\}\subset K_0(C^*(G))\]
	is a subgroup. 
\end{lemma}
\begin{proof}
	We first prove that $S$ is closed under addition.
	Let $O_1,O_2\subset G^{(0)}$ be compact open subsets. 
	By paradoxical comparison, there are compact open bisections $U_1,U_2\subset G$ such that $s(U_1)= s(U_2)=O_1\cup O_2$ and $r(U_1), r(U_2)\subset O_1\cup O_2$ are disjoint. Using \autoref{lem-elementary-K}, we get 
	\[[1_{O_1}]+[1_{O_2}]=[1_{s(U_1O_1)}]+[1_{s(U_2O_2)}]=[1_{r(U_1O_1)}]+[1_{r(U_2O_2)}]=[1_{r(U_1O_1)\cup r(U_2O_2)}].\]
	Next, we prove that $S$ is closed under inversion. 
	Let $O\subset G^{(0)}$ be a compact open subset. 
	By paradoxical comparison, there are compact open bisections $U_1,U_2\subset G$ such that $s(U_1)= s(U_2)=O$ and $r(U_1), r(U_2)\subset O$ are disjoint. Using \autoref{lem-elementary-K}, we get 
	\[[1_O]=[1_{r(U_1)}]+[1_{r(U_2)}]+[1_{O\setminus (r(U_1)\cup r(U_2))}]=[1_O]+[1_O]+[1_{O\setminus (r(U_1)\cup r(U_2))}],\]
	and thus 
	\[-[1_O]=[1_{O\setminus (r(U_1)\cup r(U_2))}].\]
\end{proof}

We briefly recall the construction and notation from \cite{Wu2025}. Wu considers a certain row-finite directed tree $X_{\mathcal G_+,x}$ equipped with an action of a discrete group $\Gamma$. The boundary $\partial X_{\mathcal G_+,x}$ is defined as equivalence classes of infinite directed paths $e_1e_2e_3\dotsb$ in $X_{\mathcal G_+,x}$ with $r(e_{i+1})=s(e_i)$ for all $i$ where two paths are equivalent if they agree after deleting finite initial segments. The basic compact open subsets of $\partial X_{\mathcal G_+,x}$ are given by 
\[Z_\partial (v)=\{[e_1e_2e_3\dotsb]\mid r(e_1)=v\},\]
for vertices $v\in X_{\mathcal G_+,x}$ (see also \cite[{\S}2.3]{Brownlowe2025}). 
We need the following elementary lemma:
\begin{lemma}\label{lem-cylinder}
	Every compact open subset of $\partial X_{\mathcal G_+,x}$ can be written as a finite disjoint union of sets of the form $Z_{\partial}(v)$.
\end{lemma}
\begin{proof}
	Since every compact open subset of $\partial X_{\mathcal G_+,x}$ can be covered by finitely many sets of the form $Z_\partial (v)$, it suffices to show that sets of the form $Z_\partial (v)$ are stable under intersection and relative complements.
	Let $v,w\in X_{\mathcal G_+,x}$ be vertices. 
	It follows from \cite[Proposition 4.4]{Brownlowe2025} that 
	$Z_\partial(v)\cap Z_\partial(w)$
	is either empty or given by $Z_\partial(v\vee w)$ where $v\vee w$ is the vertex with the minimal distance to $v$ (and $w$) which has paths to both $v$ and $w$. 
	
	We proceed by considering the set $Z_\partial(v)\setminus Z_\partial (w)$. Without loss of generality, we can assume that $Z_\partial(v)\cap Z_\partial(w)=Z_\partial(v\vee w)$. It follows that 
	\begin{equation}\label{eq-paths}
	Z_\partial(v)\setminus Z_\partial (w)=\bigsqcup_e Z_\partial (s(e)),
	\end{equation}
	where $e$ runs over all edges satisfying $r(e)\in \{r(e_1),\dotsc,r(e_n)\}$ and $e_1\dotsb e_n$ is the unique directed path with $r(e_1)=v$ and $s(e_n)=v\vee w$ (see \autoref{fig-paths}). 

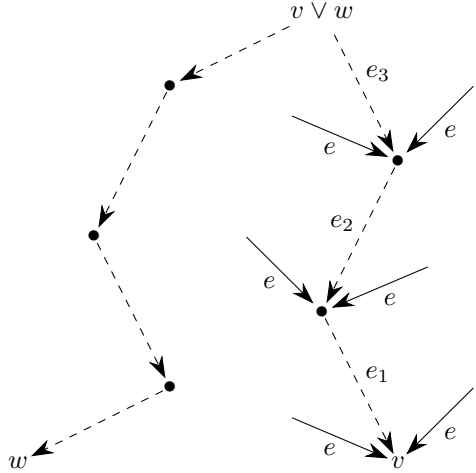
\begin{figure}[H]
\begin{tikzpicture}[
    >={Stealth[length=3mm, width=2mm]},
    every node/.style={ minimum size=0mm, inner sep=0pt}
]
\node (u0) at (0,0) {$v\vee w$};
\node (u1) at (-2,-1){$\bullet$};
\node (u2) at (-3,-3){$\bullet$};
\node (u3) at (-2,-5){$\bullet$};
\node (u4) at (-4,-6){$w$};
\node (u5) at (1,-2){$\bullet$};
\node (u6) at (0,-4){$\bullet$};
\node (u7) at (1,-6){$v$};

\node (p1) at (-0.4,-1.4){};
\node (p2) at (2,-1){};
\node (p3) at (-1,-3){};
\node (p4) at (1.4,-3.4){};
\node (p5) at (-0.4,-5.4){};
\node (p6) at (2,-5){};

 \draw[dashed,->, shorten <=7pt, shorten >=1pt] 
 	(u0) to  node[above right, outer sep=2pt]{}(u1);
 \draw[dashed,->, shorten >=1pt ] 
 	(u1) to  node[above right, outer sep=2pt]{}(u2);
 \draw[dashed,-> , shorten >=1pt] 
 	(u2) to  node[above right, outer sep=2pt]{}(u3);
 \draw[dashed,-> , shorten >=1pt] 
 	(u3) to  node[above right, outer sep=2pt]{}(u4);

 \draw[dashed,-> , shorten <=7pt,  shorten >=1pt ] 
 	(u0) to  node[above right, outer sep=2pt]{$e_3$}(u5);
 \draw[dashed,->,  shorten >=1pt  ] 
 	(u5) to  node[above left, outer sep=2pt]{$e_2$}(u6);
 \draw[dashed,->, shorten >=1pt ] 
 	(u6) to  node[above right, outer sep=2pt]{$e_1$}(u7);

\draw[ ->, shorten >=1pt  ] 
 	(p1) to  node[below left , outer sep=2pt]{$e$}(u5);
	\draw[ ->, shorten >=1pt  ] 
 	(p2) to  node[below right, outer sep=2pt]{$e$}(u5);
\draw[ ->, shorten >=1pt  ] 
 	(p3) to  node[below left, outer sep=2pt]{$e$}(u6);
\draw[ -> , shorten >=1pt ] 
 	(p4) to  node[below right, outer sep=2pt]{$e$}(u6);
	\draw[ ->, shorten >=1pt  ] 
 	(p5) to  node[below left, outer sep=2pt]{$e$}(u7);
	\draw[ ->, shorten >=1pt  ] 
 	(p6) to  node[below right, outer sep=2pt]{$e$}(u7);
\end{tikzpicture}
\caption{An illustration of the edges $e$ appearing in \eqref{eq-paths}.}
\label{fig-paths}
\end{figure}
\end{proof}

\begin{theorem}\label{thm-kirchberg}
	Let $A$ be a unital UCT Kirchberg algebra. Then there is a discrete group $\Gamma$ and a partial action $\Gamma\curvearrowright X$ on the Cantor set with clopen domains such that $A\cong C^*(\Gamma\ltimes X)$. 
\end{theorem}
\begin{proof}
	As pointed out above, we use the construction and notation from \cite{Wu2025}.
	Let $B \coloneqq A\otimes \mathbb{K}$ be the stabilization of $A$, which is a stable UCT Kirchberg algebra with $K_*(B) \cong K_*(A)$.
	By \cite[Theorem~6.2]{Wu2025}, the $K$-theory of $B$ can be realized by the combinatorial $C^*$-algebra of a directed graph of groups. 
	Following \cite[Construction~6.9]{Wu2025}, this combinatorial data gives rise to a discrete group $\Gamma$ acting on a directed tree $X_{\mathcal G_+,x}$ with an associated boundary action $\Gamma\curvearrowright \partial X_{\mathcal G_+,x}$ such that the global crossed product $C^*(\Gamma\ltimes \partial X_{\mathcal G_+,x})$ is Morita equivalent (and thus isomorphic, by stability) to this combinatorial $C^*$-algebra. 
	Consequently, $C^*(\Gamma\ltimes \partial X_{\mathcal G_+,x})$ is a stable UCT Kirchberg algebra with 
	\[K_*(C^*(\Gamma\ltimes \partial X_{\mathcal G_+,x}))\cong K_*(B) \cong K_*(A).\]
	We claim that
	\begin{claim}
	There is a compact open subset $Y\subset \partial X_{\mathcal G_+,x}$ such that the above isomorphism identifies $[1_Y]\in K_0(C^*(\Gamma\ltimes \partial X_{\mathcal G_+,x}))$ with $[1_A]\in K_0(A)$.
	\end{claim} 
	We first finish the proof of the theorem assuming that the claim is true.
	Let $\Gamma\curvearrowright Y$ be the partial action given by restricting the action $\Gamma\curvearrowright \partial X_{\mathcal G_+,x}$ to the domains $D_\gamma\coloneqq Y\cap \gamma(Y)\subset Y$ for $\gamma\in \Gamma$. 
	The canonical $*$-isomorphism $C_c(\Gamma\ltimes Y)\cong 1_Y C_c(\Gamma\ltimes \partial X_{\mathcal G_+,x})1_Y$ extends to a surjective $*$-homomorphism 
	\[\phi\colon C^*(\Gamma\ltimes Y)\twoheadrightarrow 1_Y C^*(\Gamma\ltimes \partial X_{\mathcal G_+,x})1_Y.\]
	As noted in \cite[Proposition~6.6]{Wu2025}, the action $\Gamma\curvearrowright \partial X_{\mathcal G_+,x}$ is minimal. Since $Y$ is a non-empty open set, we have $\partial X_{\mathcal G_+,x} = \bigcup_{\gamma\in \Gamma} \gamma(Y)$, which implies that the global action on $\partial X_{\mathcal G_+,x}$ is the enveloping action of the partial action $\Gamma\curvearrowright Y$. It then follows from \cite[Theorem~4.18]{Abadie2004} that $\phi$ is an isomorphism. Moreover, the fact that the translates of $Y$ cover the entire space implies that $1_Y$ is a full projection in $C^*(\Gamma\ltimes \partial X_{\mathcal G_+,x})$.

	Furthermore, by the proof of \cite[Proposition~6.14]{Wu2025}, $\Gamma\ltimes \partial X_{\mathcal G_+,x}$ is topologically principal and amenable. 
	Being a reduction of the original groupoid, $\Gamma\ltimes Y=(\Gamma\ltimes \partial X_{\mathcal G_+,x})(Y)$ enjoys these properties as well, so that $C^*(\Gamma\ltimes Y)$ is simple by \cite{Brown2014}.
	Being a corner in a UCT Kirchberg algebra, $C^*(\Gamma\ltimes Y)$ is thus a UCT Kirchberg algebra itself (see for instance \cite{Roerdam2002}).
	By \autoref{lem-full-corner}, the map
	\[C^*(\Gamma\ltimes Y)\xrightarrow[\cong]{\phi} 1_Y C^*(\Gamma\ltimes \partial X_{\mathcal G_+,x})1_Y\hookrightarrow C^*(\Gamma\ltimes \partial X_{\mathcal G_+,x})\]
	induces an isomorphism 
	\[K_*(C^*(\Gamma\ltimes Y))\cong K_*(C^*(\Gamma\ltimes \partial X_{\mathcal G_+,x}))\cong K_*(A),\]
	which by construction identifies $[1_Y]\in K_0(C^*(\Gamma\ltimes Y))$ with $[1_A]\in K_0(A)$. 
	It follows from the Kirchberg--Phillips classification theorem \cite{Kirchberg,Phillips2000} that $C^*(\Gamma\ltimes Y)\cong A$.

	We proceed with proving the claim.
	Note that \cite[Remark~6.8, Proposition~6.11(2)]{Wu2025} combined with the proofs of \cite[Proposition~6.5]{Wu2025} and \cite[Theorem~4.19]{Brownlowe2025} imply that $K_0(C^*(\Gamma\ltimes \partial X_{\mathcal G_+,x}))$ is generated by classes of the form $[1_{Z_\partial (v)}]$ where $v$ ranges over all vertices of $\partial X_{\mathcal G_+,x}$.
	In view of \autoref{lem-paradoxical}, this will imply the claim once we show that $\Gamma\ltimes \partial X_{\mathcal G_+,x}$ has paradoxical comparison. 
	Since every compact open subset of $\partial X_{\mathcal G_+,x}$ is a disjoint union of cylinder sets $Z_\partial (v)$ for vertices $v\in X_{\mathcal G_+,x}$ (see \autoref{lem-cylinder} above), it suffices to show that for every such $Z_\partial (v)$, there are group elements $\gamma_1,\gamma_2\in \Gamma$ such that $\gamma_1(Z_\partial (v))$ and $\gamma_2(Z_\partial (v))$ are disjoint subsets of $Z_\partial (v)$.
	This in turn follows from the fact that the quotient graph $\Gamma\backslash X_{\mathcal G_+,x}$ contains at least two independent\footnote{We call two loops $\eta$ and $\eta'$ \emph{independent} if none of the two is an extension of the other, i.e. $\eta\not=\eta'\nu$ and $\eta'\not=\eta\nu$ for all paths $\nu$.} loops at every vertex (see \cite[Remark~6.13]{Wu2025}). Indeed, two independent loops at the image of $v$ in $\Gamma\backslash X_{\mathcal G_+,x}$ lift to directed paths in the tree $X_{\mathcal G_+,x}$ from $v$ to $\gamma_1 v$ and $\gamma_2 v$ for some $\gamma_1, \gamma_2 \in \Gamma$. Since the tree is directed, the corresponding cylinder sets $Z_\partial(\gamma_1 v) = \gamma_1(Z_\partial(v))$ and $Z_\partial(\gamma_2 v) = \gamma_2(Z_\partial(v))$ are disjoint and both are strictly contained in $Z_\partial(v)$.

	Finally, we observe that $Y$ is homeomorphic to the Cantor set: it is compact, second countable, and totally disconnected by construction, and it has no isolated points because the groupoid $\Gamma\ltimes Y$ has paradoxical comparison.
\end{proof}

\begin{rem}
	We would like to point out that a construction similar to the one from \cite{Wu2025} that we employed here was recently used in \cite{Ma2026} to produce many examples of unital UCT Kirchberg algebras which arise as crossed products of global group actions on compact spaces. We do not know if the examples from \cite{Ma2026} exhaust all UCT Kirchberg algebras. 
\end{rem}

\bibliography{references.bib}

@Article{AnantharamanDelaroche,
  author        = {Claire Anantharaman--Delaroche},
  title         = {Groupoid exactness and the weak containment problem},
  year          = {2026},
  note          = {Preprint},
  archiveprefix = {arXiv},
  eprint        = {1605.05117},
  primaryclass  = {math.OA},
}

@article {Abadie2004,
    AUTHOR = {Abadie, Fernando},
     TITLE = {On partial actions and groupoids},
   JOURNAL = {Proc. Amer. Math. Soc.},
  FJOURNAL = {Proceedings of the American Mathematical Society},
    VOLUME = {132},
      YEAR = {2004},
    NUMBER = {4},
     PAGES = {1037--1047},
      ISSN = {0002-9939,1088-6826},
   MRCLASS = {46L05 (46L55)},
  MRNUMBER = {2045419},
MRREVIEWER = {Qing\ Xiang\ Xu},
       DOI = {10.1090/S0002-9939-03-07300-3},
       URL = {https://doi.org/10.1090/S0002-9939-03-07300-3},
}

@Article{Higson2002,
  author     = {Higson, Nigel and Lafforgue, Vincent and Skandalis, Georges},
  journal    = {Geom. Funct. Anal.},
  title      = {Counterexamples to the {B}aum-{C}onnes conjecture},
  year       = {2002},
  issn       = {1016-443X,1420-8970},
  number     = {2},
  pages      = {330--354},
  volume     = {12},
  doi        = {10.1007/s00039-002-8249-5},
  fjournal   = {Geometric and Functional Analysis},
  mrclass    = {19K56 (22A22 46L80 46L85 58J22)},
  mrnumber   = {1911663},
  mrreviewer = {Alain\ Valette},
  url        = {https://doi.org/10.1007/s00039-002-8249-5},
}

@Article{Alekseev2018,
  author   = {Alekseev, Vadim and Finn-Sell, Martin},
  journal  = {Int. Math. Res. Not. IMRN},
  title    = {Non-amenable principal groupoids with weak containment},
  year     = {2018},
  issn     = {1073-7928,1687-0247},
  number   = {8},
  pages    = {2332--2340},
  doi      = {10.1093/imrn/rnw305},
  fjournal = {International Mathematics Research Notices. IMRN},
  mrclass  = {22E47 (20L05)},
  mrnumber = {3801485},
  url      = {https://doi.org/10.1093/imrn/rnw305},
}

@Article{Willett2015,
  author     = {Willett, Rufus},
  journal    = {M\"unster J. Math.},
  title      = {A non-amenable groupoid whose maximal and reduced {$C^*$}-algebras are the same},
  year       = {2015},
  issn       = {1867-5778,1867-5786},
  number     = {1},
  pages      = {241--252},
  volume     = {8},
  doi        = {10.17879/65219671638},
  fjournal   = {M\"unster Journal of Mathematics},
  mrclass    = {46L55 (22A22 22D25 46L05)},
  mrnumber   = {3549528},
  mrreviewer = {Jean\ N.\ Renault},
  url        = {https://doi.org/10.17879/65219671638},
}

@article {CFH-Kirchberg,
    AUTHOR = {Clark, Lisa Orloff and Fletcher, James and an Huef, Astrid},
     TITLE = {All classifiable {K}irchberg algebras are {$C^*$}-algebras of
              ample groupoids},
   JOURNAL = {Expo. Math.},
  FJOURNAL = {Expositiones Mathematicae},
    VOLUME = {38},
      YEAR = {2020},
    NUMBER = {4},
     PAGES = {559--565},
      ISSN = {0723-0869,1878-0792},
   MRCLASS = {46L05 (46L35)},
  MRNUMBER = {4177957},
       DOI = {10.1016/j.exmath.2019.06.001},
       URL = {https://doi.org/10.1016/j.exmath.2019.06.001},
}

@Book{Roe2003,
  author     = {Roe, John},
  publisher  = {American Mathematical Society, Providence, RI},
  title      = {Lectures on coarse geometry},
  year       = {2003},
  isbn       = {0-8218-3332-4},
  series     = {University Lecture Series},
  volume     = {31},
  doi        = {10.1090/ulect/031},
  mrclass    = {53C24 (20F65 46L05 51K05 54E25)},
  mrnumber   = {2007488},
  mrreviewer = {Jean-Louis\ Tu},
  pages      = {viii+175},
  url        = {https://doi.org/10.1090/ulect/031},
}

@Book{Renault1980,
  author     = {Renault, Jean},
  publisher  = {Springer, Berlin},
  title      = {A groupoid approach to {$C^*$}-algebras},
  year       = {1980},
  isbn       = {3-540-09977-8},
  series     = {Lecture Notes in Mathematics},
  volume     = {793},
  doi        = {10.1007/BFb0091072},
  mrclass    = {46Lxx (22D25 22D40)},
  mrnumber   = {584266},
  mrreviewer = {A.\ K.\ Seda},
  pages      = {ii+160},
}

@Book{AnantharamanDelaroche2000,
  author     = {Anantharaman-Delaroche, Claire and Renault, Jean},
  publisher  = {L'Enseignement Math\'ematique, Geneva},
  title      = {Amenable groupoids},
  year       = {2000},
  isbn       = {2-940264-01-5},
  note       = {With a foreword by Georges Skandalis and Appendix B by E. Germain},
  series     = {Monographies de L'Enseignement Math\'ematique [Monographs of L'Enseignement Math\'ematique]},
  volume     = {36},
  mrclass    = {22A22 (22D25 43A07 46L05 46L10 46L80)},
  mrnumber   = {1799683},
  mrreviewer = {Robert\ S.\ Doran},
  pages      = {196},
}

@Article{Barlak-Li-I,
  author     = {Barlak, Sel\c{c}uk and Li, Xin},
  journal    = {Adv. Math.},
  title      = {Cartan subalgebras and the {UCT} problem},
  year       = {2017},
  issn       = {0001-8708,1090-2082},
  pages      = {748--769},
  volume     = {316},
  doi        = {10.1016/j.aim.2017.06.024},
  fjournal   = {Advances in Mathematics},
  mrclass    = {46L10},
  mrnumber   = {3672919},
  mrreviewer = {Xiao\ Chen},
  url        = {https://doi.org/10.1016/j.aim.2017.06.024},
}

@Book{Exel2017,
  author     = {Exel, Ruy},
  publisher  = {American Mathematical Society, Providence, RI},
  title      = {Partial dynamical systems, {F}ell bundles and applications},
  year       = {2017},
  isbn       = {978-1-4704-3785-5},
  series     = {Mathematical Surveys and Monographs},
  volume     = {224},
  doi        = {10.1090/surv/224},
  mrclass    = {46L55 (16S35 16S40 37A55 46-02 46L45)},
  mrnumber   = {3699795},
  mrreviewer = {Fernando\ Abadie},
  pages      = {vi+321},
  url        = {https://doi.org/10.1090/surv/224},
}

@Article{Gardella2023,
  author     = {Gardella, Eusebio and Geffen, Shirly and Kranz, Julian and Naryshkin, Petr},
  journal    = {J. Reine Angew. Math.},
  title      = {Classifiability of crossed products by nonamenable groups},
  year       = {2023},
  issn       = {0075-4102,1435-5345},
  pages      = {285--312},
  volume     = {797},
  doi        = {10.1515/crelle-2023-0012},
  fjournal   = {Journal f\"ur die Reine und Angewandte Mathematik. [Crelle's Journal]},
  mrclass    = {46L35 (20F65 37A55 46L55 46L80)},
  mrnumber   = {4565952},
  mrreviewer = {Jia-jie\ Hua},
  url        = {https://doi.org/10.1515/crelle-2023-0012},
}

@Article{deCastroKang2024,
  author        = {Gilles G. de Castro and Eun Ji Kang},
  title         = {A Categorical Interpretation of Continuous Orbit Equivalence for Partial Dynamical Systems},
  year          = {2024},
  note          = {Preprint},
  archiveprefix = {arXiv},
  eprint        = {2412.03813},
  primaryclass  = {math.OA},
}

@Article{Exel-Starling,
  author     = {Exel, Ruy and Starling, Charles},
  journal    = {J. Operator Theory},
  title      = {Self-similar graph {$C^*$}-algebras and partial crossed products},
  year       = {2016},
  issn       = {0379-4024,1841-7744},
  number     = {2},
  pages      = {299--317},
  volume     = {75},
  doi        = {10.7900/jot.2015mar04.2072},
  fjournal   = {Journal of Operator Theory},
  mrclass    = {46L55 (20F65 20M18 22A22 46L35)},
  mrnumber   = {3509131},
  mrreviewer = {J\'an\ \v Spakula},
  url        = {https://doi.org/10.7900/jot.2015mar04.2072},
}

@misc{Steinberg2026,
  author       = {Benjamin Steinberg},
  title        = {Partial Actions of Free Groups and Groupoid Homology},
  year         = {2026},
  eprint       = {2602.15170},
  archivePrefix= {arXiv},
  primaryClass = {math.RA},
  url          = {https://arxiv.org/abs/2602.15170},
  note         = {Preprint}
}

@Article{Kumjian2000,
  author     = {Kumjian, Alex and Pask, David},
  journal    = {New York J. Math.},
  title      = {Higher rank graph {$C^*$}-algebras},
  year       = {2000},
  issn       = {1076-9803},
  pages      = {1--20},
  volume     = {6},
  fjournal   = {New York Journal of Mathematics},
  mrclass    = {46L35 (46L45)},
  mrnumber   = {1745529},
  mrreviewer = {Michael\ Frank},
  url        = {http://nyjm.albany.edu:8000/j/2000/6_1.html},
}

@Article{Renault:Cartan,
  author     = {Renault, Jean},
  journal    = {Irish Math. Soc. Bull.},
  title      = {Cartan subalgebras in {$C^*$}-algebras},
  year       = {2008},
  issn       = {0791-5578},
  number     = {61},
  pages      = {29--63},
  fjournal   = {Irish Mathematical Society Bulletin},
  mrclass    = {46L85 (37B99)},
  mrnumber   = {2460017},
  mrreviewer = {Benton\ L.\ Duncan},
  url        = {https://www.maths.tcd.ie/pub/ims/bull61/S6101.pdf},
}

@Article{Deaconu1995,
  author     = {Deaconu, Valentin},
  journal    = {Trans. Amer. Math. Soc.},
  title      = {Groupoids associated with endomorphisms},
  year       = {1995},
  issn       = {0002-9947,1088-6850},
  number     = {5},
  pages      = {1779--1786},
  volume     = {347},
  doi        = {10.2307/2154972},
  fjournal   = {Transactions of the American Mathematical Society},
  mrclass    = {46L55 (19K99 20L15 46L80 54H20 57N10)},
  mrnumber   = {1233967},
  mrreviewer = {Alain\ Valette},
  url        = {https://doi.org/10.2307/2154972},
}

@Article{Brodzki2007,
  author     = {Brodzki, Jacek and Niblo, Graham A. and Wright, Nick J.},
  journal    = {J. Lond. Math. Soc. (2)},
  title      = {Property {A}, partial translation structures, and uniform embeddings in groups},
  year       = {2007},
  issn       = {0024-6107,1469-7750},
  number     = {2},
  pages      = {479--497},
  volume     = {76},
  doi        = {10.1112/jlms/jdm066},
  fjournal   = {Journal of the London Mathematical Society. Second Series},
  mrclass    = {46L85 (20F65 54E35)},
  mrnumber   = {2363428},
  mrreviewer = {Narutaka\ Ozawa},
  url        = {https://doi.org/10.1112/jlms/jdm066},
}

@Article{AnantharamanDelaroche1997,
  author     = {Anantharaman-Delaroche, Claire},
  journal    = {Bull. Soc. Math. France},
  title      = {Purely infinite {$C^*$}-algebras arising from dynamical systems},
  year       = {1997},
  issn       = {0037-9484,2102-622X},
  number     = {2},
  pages      = {199--225},
  volume     = {125},
  fjournal   = {Bulletin de la Soci\'et\'e{} Math\'ematique de France},
  mrclass    = {46L55 (46L05 54H20 58F03)},
  mrnumber   = {1478030},
  mrreviewer = {Judith\ A.\ Packer},
  url        = {http://www.numdam.org/item?id=BSMF_1997__125_2_199_0},
}

@Article{Exel1994,
  author    = {Exel, Ruy},
  journal   = {Journal of functional analysis},
  title     = {Circle actions on {$C^*$}-algebras, partial automorphisms, and a generalized {P}imsner-{V}oiculescu exact sequence},
  year      = {1994},
  number    = {2},
  pages     = {361--401},
  volume    = {122},
  doi       = {10.1006/jfan.1994.1073},
  publisher = {Elsevier},
}

@article {Li:Classifiable-Cartan,
    AUTHOR = {Li, Xin},
     TITLE = {Every classifiable simple {$\rm C^*$}-algebra has a {C}artan
              subalgebra},
   JOURNAL = {Invent. Math.},
  FJOURNAL = {Inventiones Mathematicae},
    VOLUME = {219},
      YEAR = {2020},
    NUMBER = {2},
     PAGES = {653--699},
      ISSN = {0020-9910,1432-1297},
   MRCLASS = {46L35 (22A22)},
  MRNUMBER = {4054809},
MRREVIEWER = {James\ Gabe},
       DOI = {10.1007/s00222-019-00914-0},
       URL = {https://doi.org/10.1007/s00222-019-00914-0},
}

@Article{Strickland2009,
  author  = {Strickland, Neil P.},
  journal = {preprint},
  title   = {The category of CGWH spaces},
  year    = {2009},
  url     = {https://ncatlab.org/nlab/files/StricklandCGHWSpaces.pdf},
}

@Article{Sako2020,
  author    = {Sako, Hiroki},
  journal   = {Journal of the London Mathematical Society},
  title     = {Finite-dimensional approximation properties for uniform {R}oe algebras},
  year      = {2020},
  number    = {2},
  pages     = {623--644},
  volume    = {102},
  doi       = {10.1112/jlms.12330},
  publisher = {Wiley Online Library},
}

@Article{Yu2000,
  author    = {Yu, Guoliang},
  journal   = {Inventiones mathematicae},
  title     = {The coarse {B}aum--{C}onnes conjecture for spaces which admit a uniform embedding into {H}ilbert space},
  year      = {2000},
  number    = {1},
  pages     = {201--240},
  volume    = {139},
  doi       = {10.1007/s002229900032},
  publisher = {Springer},
}

@Article{Ozawa2000,
  author    = {Ozawa, Narutaka},
  journal   = {Comptes Rendus de l'Acad{\'e}mie des Sciences-Series I-Mathematics},
  title     = {Amenable actions and exactness for discrete groups},
  year      = {2000},
  number    = {8},
  pages     = {691--695},
  volume    = {330},
  doi       = {10.1016/S0764-4442(00)00248-2},
  publisher = {Elsevier},
}

@Article{Li2017,
  author     = {Li, Xin},
  journal    = {Int. Math. Res. Not. IMRN},
  title      = {Partial transformation groupoids attached to graphs and semigroups},
  year       = {2017},
  issn       = {1073-7928,1687-0247},
  number     = {17},
  pages      = {5233--5259},
  doi        = {10.1093/imrn/rnw166},
  fjournal   = {International Mathematics Research Notices. IMRN},
  mrclass    = {37A20 (05C25 37A55)},
  mrnumber   = {3694599},
  mrreviewer = {Siming\ Tu},
  url        = {https://doi.org/10.1093/imrn/rnw166},
}

@Article{Spielberg2007,
  author    = {Spielberg, Jack},
  journal   = {Journal of Operator Theory},
  title     = {Graph-based models for {K}irchberg algebras},
  year      = {2007},
  pages     = {347--374},
  publisher = {JSTOR},
  url       = {https://jot.theta.ro/jot/archive/2007-057-002/2007-057-002-007.pdf},
}

@Article{Exel2011,
  author     = {Exel, Ruy and an Huef, Astrid and Raeburn, Iain},
  journal    = {Indiana Univ. Math. J.},
  title      = {Purely infinite simple {$C^*$}-algebras associated to integer dilation matrices},
  year       = {2011},
  issn       = {0022-2518,1943-5258},
  number     = {3},
  pages      = {1033--1058},
  volume     = {60},
  doi        = {10.1512/iumj.2011.60.4331},
  fjournal   = {Indiana University Mathematics Journal},
  mrclass    = {46L05 (46L55)},
  mrnumber   = {2985865},
  mrreviewer = {Xin\ Li},
  url        = {https://doi.org/10.1512/iumj.2011.60.4331},
}

@Article{Phillips2000,
  author     = {Phillips, N. Christopher},
  journal    = {Doc. Math.},
  title      = {A classification theorem for nuclear purely infinite simple {$C^*$}-algebras},
  year       = {2000},
  issn       = {1431-0635,1431-0643},
  pages      = {49--114},
  volume     = {5},
  doi        = {10.4171/DM/75},
  fjournal   = {Documenta Mathematica},
  mrclass    = {46L05 (19K56 46L35 46L80)},
  mrnumber   = {1745197},
  mrreviewer = {Mikael\ R\o rdam},
}

@Article{Kirchberg,
  author  = {Eberhard Kirchberg},
  journal = {Preprint},
  title   = {The classification of purely infinite ${C^*}$-algebras using {K}asparov’s theory.},
  year    = {1994},
}

@Article{Pask2004,
  author     = {Pask, David and Quigg, John and Raeburn, Iain},
  journal    = {New York J. Math.},
  title      = {Fundamental groupoids of {$k$}-graphs},
  year       = {2004},
  issn       = {1076-9803},
  pages      = {195--207},
  volume     = {10},
  fjournal   = {New York Journal of Mathematics},
  mrclass    = {05C20 (14H30 46L05)},
  mrnumber   = {2114786},
  mrreviewer = {Alain\ Valette},
  url        = {http://nyjm.albany.edu:8000/j/2004/10_195.html},
}

@Book{Bridson1999,
  author     = {Bridson, Martin R. and Haefliger, Andr\'e},
  publisher  = {Springer-Verlag, Berlin},
  title      = {Metric spaces of non-positive curvature},
  year       = {1999},
  isbn       = {3-540-64324-9},
  series     = {Grundlehren der mathematischen Wissenschaften [Fundamental Principles of Mathematical Sciences]},
  volume     = {319},
  doi        = {10.1007/978-3-662-12494-9},
  mrclass    = {53C23 (20F65 53C70 57M07)},
  mrnumber   = {1744486},
  mrreviewer = {Athanase\ Papadopoulos},
  pages      = {xxii+643},
  url        = {https://doi.org/10.1007/978-3-662-12494-9},
}

@Article{Tu1999,
  author     = {Tu, Jean-Louis},
  journal    = {$K$-Theory},
  title      = {La conjecture de {B}aum-{C}onnes pour les feuilletages moyennables},
  year       = {1999},
  issn       = {0920-3036,1573-0514},
  number     = {3},
  pages      = {215--264},
  volume     = {17},
  doi        = {10.1023/A:1007744304422},
  fjournal   = {$K$-Theory. An Interdisciplinary Journal for the Development, Application, and Influence of $K$-Theory in the Mathematical Sciences},
  mrclass    = {19K35 (46L85 58J22)},
  mrnumber   = {1703305},
  mrreviewer = {Hiroshi\ Takai},
  url        = {https://doi.org/10.1023/A:1007744304422},
}

@InProceedings{Winter2018,
  author     = {Winter, Wilhelm},
  booktitle  = {Proceedings of the {I}nternational {C}ongress of {M}athematicians---{R}io de {J}aneiro 2018. {V}ol. {III}. {I}nvited lectures},
  title      = {Structure of nuclear {$C^*$}-algebras: from quasidiagonality to classification and back again},
  year       = {2018},
  pages      = {1801--1823},
  publisher  = {World Sci. Publ., Hackensack, NJ},
  doi        = {10.1142/9789813272880_0118},
  isbn       = {978-981-3272-92-7; 978-981-3272-87-3},
  mrclass    = {46L05 (43A07 46L35)},
  mrnumber   = {3966830},
  mrreviewer = {Maria\ Grazia\ Viola},
}

@InProceedings{White2023,
  author       = {White, Stuart},
  booktitle    = {International Congress of Mathematicians},
  title        = {Abstract classification theorems for amenable ${C^*}$-algebras},
  year         = {2023},
  organization = {European Mathematical Society-EMS-Publishing House GmbH},
  pages        = {3314--3338},
  doi          = {10.4171/ICM2022/183},
}

@Article{Rosenberg1987,
  author     = {Rosenberg, Jonathan and Schochet, Claude},
  journal    = {Duke Math. J.},
  title      = {The {K}\"unneth theorem and the universal coefficient theorem for {K}asparov's generalized {$K$}-functor},
  year       = {1987},
  issn       = {0012-7094,1547-7398},
  number     = {2},
  pages      = {431--474},
  volume     = {55},
  doi        = {10.1215/S0012-7094-87-05524-4},
  fjournal   = {Duke Mathematical Journal},
  mrclass    = {46L80 (19K33 46M20 58G12)},
  mrnumber   = {894590},
  mrreviewer = {Thierry\ Fack},
  url        = {https://doi.org/10.1215/S0012-7094-87-05524-4},
}

@Article{McClanahan1995,
  author     = {McClanahan, Kevin},
  journal    = {J. Funct. Anal.},
  title      = {{$K$}-theory for partial crossed products by discrete groups},
  year       = {1995},
  issn       = {0022-1236,1096-0783},
  number     = {1},
  pages      = {77--117},
  volume     = {130},
  doi        = {10.1006/jfan.1995.1064},
  fjournal   = {Journal of Functional Analysis},
  mrclass    = {46L80 (19K56 46L55)},
  mrnumber   = {1331978},
  mrreviewer = {Terry\ A.\ Loring},
  url        = {https://doi.org/10.1006/jfan.1995.1064},
}

@Article{Buss2026,
  author   = {Buss, Alcides and Mart\'inez, Diego},
  journal  = {J. Funct. Anal.},
  title    = {Essential groupoid amenability and nuclearity of groupoid {C}*-algebras},
  year     = {2026},
  issn     = {0022-1236,1096-0783},
  number   = {11},
  pages    = {Paper No. 111427},
  volume   = {290},
  doi      = {10.1016/j.jfa.2026.111427},
  fjournal = {Journal of Functional Analysis},
  mrclass  = {20L05 (46L05 46L52)},
  mrnumber = {5040091},
  url      = {https://doi.org/10.1016/j.jfa.2026.111427},
}

@Article{Higson2001,
  author     = {Higson, Nigel and Kasparov, Gennadi},
  journal    = {Invent. Math.},
  title      = {{$E$}-theory and {$KK$}-theory for groups which act properly and isometrically on {H}ilbert space},
  year       = {2001},
  issn       = {0020-9910,1432-1297},
  number     = {1},
  pages      = {23--74},
  volume     = {144},
  doi        = {10.1007/s002220000118},
  fjournal   = {Inventiones Mathematicae},
  mrclass    = {19K35 (19L47 46L80)},
  mrnumber   = {1821144},
  mrreviewer = {Emmanuel\ C.\ Germain},
  url        = {https://doi.org/10.1007/s002220000118},
}

@Article{Gabe2024,
  author     = {Gabe, James and Szab\'o, G\'abor},
  journal    = {Acta Math.},
  title      = {The dynamical {K}irchberg-{P}hillips theorem},
  year       = {2024},
  issn       = {0001-5962,1871-2509},
  number     = {1},
  pages      = {1--77},
  volume     = {232},
  doi        = {10.4310/acta.2024.v232.n1.a1},
  fjournal   = {Acta Mathematica},
  mrclass    = {37A20 (19K35 37A55 46L55)},
  mrnumber   = {4747811},
  mrreviewer = {Yuhei\ Suzuki},
  url        = {https://doi.org/10.4310/acta.2024.v232.n1.a1},
}

@Article{Guentner2002,
  author    = {Guentner, Erik and Kaminker, Jerome},
  journal   = {Topology},
  title     = {Exactness and the {N}ovikov conjecture},
  year      = {2002},
  number    = {2},
  pages     = {411--418},
  volume    = {41},
  doi       = {10.1016/S0040-9383(00)00036-7},
  publisher = {Elsevier},
}

@Article{Connes1981,
  author   = {Connes, Alain and Feldman, Jacob and Weiss, Benjamin},
  journal  = {Ergodic Theory Dynam. Systems},
  title    = {An amenable equivalence relation is generated by a single transformation},
  year     = {1981},
  issn     = {0143-3857,1469-4417},
  number   = {4},
  pages    = {431--450},
  volume   = {1},
  doi      = {10.1017/s014338570000136x},
  fjournal = {Ergodic Theory and Dynamical Systems},
  mrclass  = {46L55 (22D40)},
  mrnumber = {662736},
  url      = {https://doi.org/10.1017/s014338570000136x},
}

@Article{Anantharaman2023,
  author  = {Claire Anantharaman-Delaroche},
  journal = {Münster J. Math.},
  title   = {Amenability, exactness and weakcontainment property for groupoids},
  year    = {2023},
  url     = {https://www.uni-muenster.de/FB10/mjm/acc/mjm-Anantharaman.pdf},
}

@Article{Exelprivate,
  author = {Ruy Exel},
  title  = {Private communication},
}

@Article{Wu2025,
  author     = {Wu, Victor},
  journal    = {J. Funct. Anal.},
  title      = {{$C^*$}-algebras associated to directed graphs of groups, and models of {K}irchberg algebras},
  year       = {2025},
  issn       = {0022-1236,1096-0783},
  number     = {3},
  pages      = {Paper No. 110740, 39},
  volume     = {288},
  doi        = {10.1016/j.jfa.2024.110740},
  fjournal   = {Journal of Functional Analysis},
  mrclass    = {46L05 (46L35 46L55)},
  mrnumber   = {4823109},
  mrreviewer = {Miho\ Mukohara},
  url        = {https://doi.org/10.1016/j.jfa.2024.110740},
}

@Article{Brownlowe2025,
  author     = {Brownlowe, Nathan and Spielberg, Jack and Thomas, Anne and Wu, Victor},
  journal    = {Trans. Amer. Math. Soc.},
  title      = {Group actions on multitrees and the {$K$}-theory of their crossed products},
  year       = {2025},
  issn       = {0002-9947,1088-6850},
  number     = {3},
  pages      = {1697--1732},
  volume     = {378},
  doi        = {10.1090/tran/9304},
  fjournal   = {Transactions of the American Mathematical Society},
  mrclass    = {46L80 (20E08)},
  mrnumber   = {4866348},
  mrreviewer = {Angel\ Rom\'an},
  url        = {https://doi.org/10.1090/tran/9304},
}

@Article{Ma2022,
  author   = {Ma, Xin},
  journal  = {Int. Math. Res. Not. IMRN},
  title    = {Purely infinite locally compact {H}ausdorff \'etale groupoids and their {$C^*$}-algebras},
  year     = {2022},
  issn     = {1073-7928,1687-0247},
  number   = {11},
  pages    = {8420--8471},
  doi      = {10.1093/imrn/rnaa360},
  fjournal = {International Mathematics Research Notices. IMRN},
  mrclass  = {46L05 (46L35)},
  mrnumber = {4425841},
  url      = {https://doi.org/10.1093/imrn/rnaa360},
}

@Article{Brown2014,
  author     = {Brown, Jonathan and Clark, Lisa Orloff and Farthing, Cynthia and Sims, Aidan},
  journal    = {Semigroup Forum},
  title      = {Simplicity of algebras associated to \'etale groupoids},
  year       = {2014},
  issn       = {0037-1912,1432-2137},
  number     = {2},
  pages      = {433--452},
  volume     = {88},
  doi        = {10.1007/s00233-013-9546-z},
  fjournal   = {Semigroup Forum},
  mrclass    = {46L05 (22A22)},
  mrnumber   = {3189105},
  mrreviewer = {Jean\ N.\ Renault},
  url        = {https://doi.org/10.1007/s00233-013-9546-z},
}

@Article{Kasparov1988,
  author     = {Kasparov, Gennadi. G.},
  journal    = {Invent. Math.},
  title      = {Equivariant {$KK$}-theory and the {N}ovikov conjecture},
  year       = {1988},
  issn       = {0020-9910,1432-1297},
  number     = {1},
  pages      = {147--201},
  volume     = {91},
  doi        = {10.1007/BF01404917},
  fjournal   = {Inventiones Mathematicae},
  mrclass    = {58G12 (19K33 19K56 46L80 46M20 53C20 57R67)},
  mrnumber   = {918241},
  mrreviewer = {Jonathan\ M.\ Rosenberg},
  url        = {https://doi.org/10.1007/BF01404917},
}

@Book{Roerdam2002,
  author    = {R{\o}rdam, Mikael and St{\o}rmer, Erling},
  publisher = {Springer-Verlag, Berlin},
  title     = {Classification of nuclear {$C^*$}-algebras. {E}ntropy in operator algebras},
  year      = {2002},
  isbn      = {3-540-42305-X},
  note      = {Operator Algebras and Non-commutative Geometry, 7},
  series    = {Encyclopaedia of Mathematical Sciences},
  volume    = {126},
  doi       = {10.1007/978-3-662-04825-2},
  mrclass   = {46Lxx (46-06)},
  mrnumber  = {1878881},
  pages     = {x+198},
  url       = {https://doi.org/10.1007/978-3-662-04825-2},
}

@Article{Ma2026,
  author        = {Xin Ma and Daxun Wang and Wenyuan Yang},
  title         = {Boundary actions of Bass-Serre Trees and the applications to $C^*$-algebras},
  year          = {2026},
  archiveprefix = {arXiv},
  eprint        = {2502.02039},
  primaryclass  = {math.OA},
}

@Article{Kerr2020,
  author     = {Kerr, David},
  journal    = {J. Eur. Math. Soc. (JEMS)},
  title      = {Dimension, comparison, and almost finiteness},
  year       = {2020},
  issn       = {1435-9855,1435-9863},
  number     = {11},
  pages      = {3697--3745},
  volume     = {22},
  doi        = {10.4171/jems/995},
  fjournal   = {Journal of the European Mathematical Society (JEMS)},
  mrclass    = {37A55 (37A20 46L35 46L55)},
  mrnumber   = {4167017},
  mrreviewer = {Bruno\ Brogni\ Uggioni},
  url        = {https://doi.org/10.4171/jems/995},
}

@Article{Brownlowe2025a,
  author   = {Brownlowe, Nathan and Kumjian, Alex and Pask, David and Sims, Aidan},
  journal  = {J. Aust. Math. Soc.},
  title    = {Embeddability of higher-rank graphs in groupoids and the structure of their {$C^*$}-algebras},
  year     = {2025},
  issn     = {1446-7887,1446-8107},
  number   = {3},
  pages    = {314--346},
  volume   = {119},
  doi      = {10.1017/S1446788725101109},
  fjournal = {Journal of the Australian Mathematical Society},
  mrclass  = {46L05 (18D99)},
  mrnumber = {5002159},
}
\bibliographystyle{alphaurl}

\end{document}